\newif\ifarxiv\arxivtrue

\PassOptionsToPackage{dvipsnames,usenames}{xcolor}
\documentclass{article}

\usepackage[accepted]{icml2024}\arxivfalse

\makeatletter

\PassOptionsToPackage{capitalize,noabbrev,nameinlink}{cleveref}
\newif\ifProofsInAppendix
\newif\ifInAppendix\InAppendixfalse
\ifarxiv
	\ProofsInAppendixfalse
	\usepackage[
		arxiv,
		thmreset=section,
		excludethm={assumption},
	]{myPreamble}
	\CreateNewTheorem{assumption}[dummythm]{Assumption}
	\setlist[itemize,1]{leftmargin=*}

	\setlistlabel[enumerate,1]{\upshape{\arabic*}.}%
	\renewcommand{\refenumeratePrefixi}{}          %
	\renewcommand{\refenumeratelabeli}{(\arabic*)} %

	\def\myampersand{}%
	\newcommand{\sep}{ \(\cdot\) }

\else
	\ProofsInAppendixtrue
	\usepackage[
		verbose=false,
		excludethm=all,
		customsettings=false,
		noload={algpseudocode,caption,subcaption},
	]{myPreamble}

	\hypersetup{%
		breaklinks,
		urlcolor=Blue,
		citecolor=Blue,
		linkcolor=Blue,
		colorlinks=true,
	}

	\def\myampersand{&}%
	\newcommand{\sep}{,\ }

	\usepackage{algorithmic}

	\theoremstyle{plain}
	\newtheorem{theorem}{Theorem}[section]
	
	\newtheorem{lemma}[theorem]{Lemma}
	\newtheorem{corollary}[theorem]{Corollary}
	\theoremstyle{definition}
	
	\newtheorem{assumption}[theorem]{Assumption}
	\theoremstyle{remark}
	\newtheorem{remark}[theorem]{Remark}

	\let\cite\citep%
	\theoremstyle{plain}
	\newtheorem{fact}[theorem]{Fact}

	\setlistlabel[enumerate,1]{\arabic*}
	\setlistlabel[enumerateq,1]{\textit{(\alph*)}}
	\setlistlabel[enumerator,1]{\textit{(\Alph*)}}
	\setlistlabel[enumeratass,1]{A\oldstylenums{\arabic*}}
	\setlistlabel[enumeratprop,1]{P\oldstylenums{\arabic*}}
	\setlist[proofitemize,1]{label=\textbullet, leftmargin=*}
	\setlist[proofitemize,2]{label=\(\diamondsuit\)}

	\renewcommand{\enumerateSuffixi}{.}
	\renewcommand{\refenumeratelabeli}{\oldstylenums{\arabic*}}
	\renewcommand{\refenumeratePrefixi}{.}

	\setlistoptions{itemsep=2pt, labelsep=3pt, leftmargin=*, align=right, before={\noindent\hspace*{0pt}}, topsep=0pt, partopsep=0pt}
	\setlistoptions[proofitemize]{itemsep=2pt, labelsep=3pt, leftmargin=*, align=right, before={\noindent\hspace*{0pt}}}

	\RegisterTheorem{assumption}{Assumption}
	\RegisterTheorem{corollary}{Corollary}
	\RegisterTheorem{definition}{Definition}
	\RegisterTheorem{fact}{Fact}
	\RegisterTheorem{lemma}{Lemma}
	\RegisterTheorem{proposition}{Proposition}
	\RegisterTheorem{remark}{Remark}
	\RegisterTheorem{theorem}{Theorem}

	\newcommand{\Require}{\REQUIRE}
	\newcommand{\State}{\STATE}
	\newcommand{\Statex}{\item[]}
	\let\algfont\textbf
\fi

	\colorlet{bodycolor}{Brown!50}
	\NewEnviron{WhereToPutThis}{{}
		\ifInAppendix
			\ifProofsInAppendix\BODY\fi
		\else
			\ifProofsInAppendix\else\unskip\bgroup\color{bodycolor}\BODY\egroup\fi
		\fi
	{} }

	\newlist{claims}{enumerate}{1}
	\setlist[claims,1]{
		label={\it Claim \ref*{\currentlabel}(\alph*):},
		ref={\ref*{\currentlabel}(\alph*)},
		topsep=0pt,
		partopsep=0pt,
		parsep=0pt,
		itemsep=3pt,
		wide=0pt,
		leftmargin=0.5cm,
	}
	\newlist{claims*}{enumerate}{1}
	\setlist[claims*,1]{
		label={\it Contradiction claim \oldstylenums{\arabic*}$^*$:},
		ref={\oldstylenums{\arabic*}$^*$},
		partopsep=0pt,
		parsep=0pt,
		itemsep=3pt,
		wide,
		labelindent=0pt,
	}

	\Crefname{claimsi}{Claim}{Claims}
	\crefname{enumeratpropi}{property}{properties}

	\crefname{ALC@line}{step}{steps}

	\newcommand{\algnamefont}[1]{\textsf{#1}}
	\NewDocumentCommand{\adaPG}{s e{^}}{%
		\IfBooleanF{#1}{\IfValueF{#2}{\hyperref[alg:adaPG]}}{%
			\algnamefont{adaPG}%
			\IfValueTF{#2}{$^{#2}$}{$^{\rp,\frac{\rp}{2}}$}%
		}%
	}%
	\NewDocumentCommand{\AdaPG}{s e{^}}{%
		\IfBooleanF{#1}{\IfValueF{#2}{\hyperref[alg:adaPG]}}{%
			\algnamefont{AdaPG}%
			\IfValueTF{#2}{$^{#2}$}{$^{\rp,\frac{\rp}{2}}$}%
		}%
	}%

	\newcommand{\NUPG}{\algnamefont{NUPG}}
	\newcommand{\FNUPG}{\algnamefont{F-NUPG}}
	\newcommand{\ACFGM}{\algnamefont{AC-FGM}}

	\edef\OLDl{\l}
	\edef\OLDL{\L}
	\DeclareMathOperator{\diam}{diam}
	\newcommand{\K}{\@ifstar{K^{\downarrow}}{K^{\uparrow}}}

	\let\oldnabla\nabla
	\renewcommand{\nabla}{\@ifstar{\widetilde{\oldnabla}}{\oldnabla}}%
	\def\innprod{\@ifstar\@innprod\@@innprod}

	\NewDocumentCommand{\U}{e{_^}}{%
		\mathcal{U}
		\IfValueT{#1}{_{#1}}
		\IfValueTF{#2}{^{#1}}{^{\rp}}
	}
	\newcommand{\X}{\mathcal{X}}
	\newcommand{\rp}{q}

	\newcommand{\newvark}[2]{\newcommand #1{\@ifstar{#2_{k+1}}{#2_k}}}

	\newcommand{\q}{%
		\ifinnew
			\nu
		\else
			\ifinold
				q
			\else
				\nu %
			\fi
		\fi
	}
	\def\c_#1{c_{{#1},\q}}
	\def\l_#1{\ell_{{#1},\q}}
	\def\L_#1{L_{{#1},\q}}
	\def\vrho_#1{\varrho_{{#1},\q}}

	\newvark{\gamk}{\gam}
	\newvark{\lamk}{\lam}
	\newvark{\Mk}{\M}
	\def\gam_#1{\gamma_{#1}}
	\def\lam_#1{\lambda_{{#1},\q}}
	\def\D_#1{D_{{#1},\q}}
	\def\M_#1{M_{{#1},\q}}

	\newvark{\rhok}{\rho}
	\newvark{\vrhok}{\vrho}
	\newvark{\Hk}{H}
	\newvark{\ck}{\c}
	\newvark{\lk}{\l}
	\newvark{\Lk}{\L}
	\newvark{\pk}{p}

	\pgfplotsset{compat=1.16}
	\pgfplotsset{
		myaxis/.style = {
			grid = major,
			grid style = {dashed, line width = .1pt, draw = gray!10},
			major grid style = {line width = .2pt, draw = gray!50},
			legend style = {at = {(0.5,1.25)}, anchor = south, legend columns = -1, font = {\algnamefont}
			},
		},
		myaxisD/.style = {
			myaxis,
			xlabel = {\# of calls to \(\linop\) and \(\linop*\)},
		},
		leftto/.style = {
			at = (#1.south east),
			xshift = 2.3cm,
			ylabel = {},
		},
		top/.style = {%
			xlabel = {},
		},
		PGaxis/.style = {
			myaxis,
			xmin = 0,
			ylabel = {$\varphi(x^k) - \min\varphi$},
		},
		PDaxis/.style = {
			myaxis,
			xmin = 0,
			ylabel = {$\|v\|$},
		},
		PDaxisD/.style = {
			myaxisD,
			xmin = 0,
			ylabel = {$\|v\|$},
		},
		gamaxis/.style = {
			myaxis,
			xlabel = {\# of iterations},
			ylabel = {$\gamma_k$},
			legend style = {at = {(0.5,1_05)},legend columns = 3},%
		},
	}

	\pgfplotsset{
		plotline/.style 2 args = {
			very thick,
			mark = {#1},
			color = {#2},
			mark size = 3pt,
			mark options = {solid},
			mark repeat = 5,
			mark phase = 1,
		},
		plotline*/.style 2 args = {
			plotline = {#1}{#2},
			dashed,
		},
		PG/.style = {%
			plotline = {o}{cyan},
		},
		PG-ls/.style = {%
			plotline* = {x}{teal},
		},
		PG-ls2/.style = {%
			plotline* = {x}{teal},
		},
		Nesterov/.style = {
			plotline = {x}{red},
		},
		aGraal/.style = {
			plotline* = {o}{violet},
		},
		MM/.style = {
			plotline = {square}{blue},
			mark size = 2.7pt,
			mark options = {dashed},
		},
		MP/.style = {
			MM
		},
		CV/.style = {
			plotline* = {diamond}{brown},
		},
		oursPG/.style = {%
			plotline = {triangle}{black},
		},
		oursPD/.style = {%
			oursPG,
		},
		oursPD+/.style = {%
			plotline* = {square}{black},
		},
	}

	\newif\ifshowold\showoldfalse
	\newif\ifshownew\shownewfalse
	\newif\ifinold\inoldfalse
	\newif\ifinnew\innewfalse
	\colorlet{oldcolor}{black!30}
	\colorlet{newcolor}{orange!70!red}

	\newcommand{\disablecolorlinks}{\def\HyColor@UseColor##1{}}

	\newcommand{\old}[1]{}
	\newcommand{\new}[1]{{\innewtrue\disablecolorlinks\color{newcolor}{}#1}}

\makeatother
\newcommand{\TheTitle}{Adaptive Proximal Gradient Methods Are Universal\ifarxiv\texorpdfstring{\\}{ }\else\ \fi Without Approximation}
\newcommand{\TheShortTitle}{Adaptive Proximal Gradient Methods Are Universal Without Approximation}
\newcommand{\TheKeywords}{%
	Convex minimization%
\sep
	proximal gradient method%
\sep
	universal methods%
\sep
	linesearch-free adaptive methods%
\sep
	local H\"older gradient continuity%
\sep
	first-order methods%
}
\newcommand{\TheAMSSubj}{%
	\amsmscLink{65K05}%
\sep
	\amsmscLink{90C06}%
\sep
	\amsmscLink{90C25}%
\sep
	\amsmscLink{90C30}%
\sep
	\amsmscLink{90C47}%
}
\newcommand{\TheFunding}{%
	Work supported by:
	the Research Foundation Flanders postdoctoral grant 12Y7622N and research projects G081222N, G033822N, and G0A0920N;
	European Union's Horizon 2020 research and innovation programme under the Marie Sk\OLDl odowska-Curie grant agreement No. 953348;
	Japan Society for the Promotion of Science (JSPS) KAKENHI grants JP21K17710 and JP24K20737.%
}

\ifarxiv\else
	\icmltitlerunning{\TheShortTitle}
\fi

\showgrayfalse
\shownewtrue
\renewcommand{\old}[1]{}
\renewcommand{\new}[1]{#1}

\begin{document}

\ifarxiv
	\title{\TheTitle\thanks{\TheFunding}}
	\author{%
		Konstantinos A. Oikonomidis\thanks{%
			\TheAddressKU.
			\{%
				\emailLink{konstantinos.oikonomidis}[@kuleuven.be],%
				\emailLink{puya.latafat}[@kuleuven.be],%
				\emailLink{emanuel.laude}[@kuleuven.be],%
				\emailLink{panos.patrinos}[@esat.kuleuven.be]%
			\}%
			\emailLink[%
				konstantinos.oikonomidis@esat.kuleuven.be,
				puya.latafat@esat.kuleuven.be,
				emanuel.laude@esat.kuleuven.be,
				panos.patrinos@esat.kuleuven.be%
			]{@esat.kuleuven.be}%
		}%
	\and
		Emanuel Laude\footnotemark[2]%
	\and
		Puya Latafat\footnotemark[2]%
	\and
		Andreas Themelis\thanks{%
			\TheAddressKUJ.
			\emailLink{andreas.themelis@ees.kyushu-u.ac.jp}%
		}%
	\and
		Panagiotis Patrinos\footnotemark[2]%
	}
	\date{}
	\maketitle
\else
	\twocolumn[%
		\icmltitle{\TheTitle}

		\begin{icmlauthorlist}
			\icmlauthor{Konstantinos Oikonomidis}{leuven}
			\icmlauthor{Emanuel Laude}{leuven}
			\icmlauthor{Puya Latafat}{leuven}
			\icmlauthor{Andreas Themelis}{kyushu}
			\icmlauthor{Panagiotis Patrinos}{leuven}
		\end{icmlauthorlist}

		\icmlaffiliation{leuven}{\TheAddressKU}
		\icmlaffiliation{kyushu}{\TheAddressKUJ}
		\icmlcorrespondingauthor{Konstantinos Oikonomidis}{konstantinos.oikonomidis@kuleuven.be}

		\icmlkeywords{\TheKeywords}

		\vskip 0.3in
	]
	\printAffiliationsAndNotice{}%
\fi

	\begin{abstract}
		We show that adaptive proximal gradient methods for convex problems are not restricted to traditional Lipschitzian assumptions.
		Our analysis reveals that a class of linesearch-free methods is still convergent under mere local H\"older gradient continuity, covering in particular continuously differentiable semi-algebraic functions.
		To mitigate the lack of local Lipschitz continuity, popular approaches revolve around $\varepsilon$-oracles and/or linesearch procedures.
		In contrast, we exploit plain H\"older inequalities not entailing any approximation, all while retaining the linesearch-free nature of adaptive schemes.
		Furthermore, we prove full sequence convergence without prior knowledge of local H\"older constants nor of the order of H\"older continuity.
		Numerical experiments make comparisons with baseline methods on diverse tasks from machine learning covering both the locally and the globally H\"older setting.
	\end{abstract}
	\ifarxiv
		\subjclass{\TheAMSSubj}
	\fi

	\section{Introduction}

		We consider composite minimization problems of the form
		\[\tag{P}\label{eq:P}
			\minimize_{x\in\R^n}\; \varphi(x) \coloneqq f(x)+g(x)
		\]
		where \(\func{f}{\R^n}{\R}\) is convex and has \emph{locally} H\"older continuous gradient of {(possibly \emph{unknown})} order \(\q\in(0,1]\), and \(\func{g}{\R^n}{\Rinf}\) is proper, lsc, and convex with easy to compute proximal mapping.

		The proximal gradient method is the de facto splitting technique for solving {the }composite problem \eqref{eq:P}.
		Under global Lipschitz continuity of \(\nabla f\){, convergence results and complexity bounds are well established.}
		Nevertheless, there exist many applications where such an assumption is not met.
		Among these{}  are mixtures of maximum likelihood models \cite{grimmer2023optimal}, classification, robust regression \cite{yang2018rsg, forsythe1972robust}, compressive sensing \cite{chartrand2008iteratively} and \(p\)-Laplacian problems on graphs \cite{hafiene2018nonlocal},{}  or the subproblems in the power augmented Lagrangian method \cite{luque1984nonlinear,oikonomidis2023power,laude2023anisotropic}.

		Although often linesearch methods are still applicable under this setting, their additional backtracking procedures can be quite costly in practice.
		In response to this, this work investigates linesearch-free adaptive methods under mere \emph{local H\"older continuity} of \(\nabla f\), covering in particular all continuously differentiable semi-algebraic functions.\footnote{%
			This is a consequence of the \OLDL ojasiewicz inequality: for a continuous semi-algebraic function \(\func{H}{\R^n}{\R^n}\), consider \(f(x,y)=\|x-y\|\) and \(g(x,y)=\|H(x)-H(y)\|\) in \cite[Cor. 2.6.7]{bochnak1998real}.
		}

		In the H\"older differentiable setting, a notable approach was introduced in the seminal works \cite{nesterov2015universal,devolder2014first} that rely on the notion of $\varepsilon$-oracles \cite[Def. 1]{devolder2014first}.
		The main idea there is to approximate the H\"older smooth term with the squared Euclidean norm, resulting in an approximate descent lemma \cite[Lem. 2]{nesterov2015universal} { that can be leveraged for a linesearch procedure}.
		More specifically, given \(\eta \in (0,1)\), some \(\gamma_0>0\), and an accuracy threshold \(\varepsilon>0\), \emph{Nesterov's universal primal gradient} (\NUPG) method \cite{nesterov2015universal} consists of computing
		\begin{subequations}\label{eq:Nesterov}
			\begin{align}
				x^{k+1}
			={} &
				\prox_{\gamk* g}\left(x^k - \gamk*\nabla f(x^k)\right),
			\label{eq:Nesterov:PG}
			\shortintertext{where \(\gamk*=2\gamk\eta^{m_k}\) and \(m_k\in\N\) is the smallest such that}
				f(x^{k+1})
			\leq{} &
				f(x^k) + \innprod{\nabla f(x^k)}{x^{k+1}-x^k}
		\ifarxiv\else
			\nonumber
			\\
		\fi
			\myampersand
				+
				\tfrac{1}{2\gamk*}\|x^{k+1} - x^k\|^2 + \tfrac12 \varepsilon.
			\label{eq:Nesterov:PG:LS}
			\shortintertext{%
				It is clear that \(\varepsilon\) is a parameter of the algorithm, a fact \ifarxiv\else which is \fi further illustrated by the convergence rate
			}
				\varphi(x^k)-\varphi(x^\star)
			\leq{} &
				\tfrac{\varepsilon}{2}
				+
				\tfrac{1}{\varepsilon}^{\frac{1-\q}{1+\q}}
				\tfrac{2L_{\q}^{\frac{2}{1+\q}}\|x^0-x^\star\|^2}{k+1},
			\label{eq:Nesterov:rate}
			\end{align}
		\end{subequations}
		where \(L_{\q}\) is a modulus of H\"older continuity of the gradient:{}
		the coefficient of the \(\nicefrac{1}{(k+1)}\) term becomes arbitrarily large as higher accuracy is demanded \(\varepsilon \to 0\).
		This approach nevertheless allows handling H\"older smooth problems in the same manner as Lipschitz smooth ones and thus implementing classical improvements such as acceleration
		\cite{nesterov2015universal, kamzolov2021universal, ghadimi2019generalized}.
		Moreover, its implementability has led to algorithms that go beyond the classical forward-backward splitting, such as primal-dual methods
		\cite{yurtsever2015universal, nesterov2021primal} and even variational inequalities \cite{stonyakin2021inexact}.

		In the context of classical majorization-min\-i\-miza\-tion, when the order and the modulo of smoothness{}  are known,
		the H\"older smoothness inequality itself has been used to generate descent without the aformentioned approximation.
		Akin to Lipschitz continuity, H\"older continuity of $\nabla f$ with constant $H$ translates into a descent lemma inequality which, after addition of $g(x)$, yields the upper bound to the cost $\varphi$
		\begin{align}\label{eq:MMHolder}
		\ifarxiv\else\nonumber\fi
			\varphi(x)
		\leq{} &
			f(x^k) + \innprod{\nabla f(x^k)}{x - x^k} + g(x)
		\ifarxiv\else
			\\
			\myampersand
		\fi
			+ \tfrac{1}{(1+\q)\lamk*}\|x-x^k\|^{1+\q},
		\end{align}
		for any $\lamk* < 1/H$, cf. \cref{thm:Lq:f:leq}.
		This was considered in \cite{bredies2008forward,guan2021forward} in a general Banach space setup as well as in \cite{yashtini2016global,bolte2023backtrack} for smooth but possibly nonconvex problems.
		With \(x^{k+1}\) denoting the minimizer of the above majorization model, the first-order optimality condition
		\begin{align} \label{eq:MM_inclusion}
		\ifarxiv\else\nonumber\fi
			0
		\in{} &
			\partial g(x^{k+1}) + \nabla f(x^k)
		\ifarxiv\else
			\\
			\myampersand
		\fi
			+ \tfrac{1}{\lamk*} \|x^{k+1}- x^k\|^{-(1-\q)}(x^{k+1} - x^k)
		\end{align}
		reveals that the resulting iterations are essentially an Euclidean proximal gradient method for a particular stepsize
		\(
			\gamk*\coloneqq\lamk*\|x^{k+1}- x^k\|^{1-\q}
		\)
		that bears an implicit dependence on the future iterate $x^{k+1}$.

		Such a majorize-minimize paradigm adopts \(\lamk*\) as an \emph{explicit} stepsize parameter, and is thus tied to (the knowledge of) the order \(\q\) of H\"older differentiability.
		Instead, we directly derive conditions on the ``Euclidean'' stepsize \(\gamk*\) and rather regard \(\lamk*\) as an \emph{implicit} parameter, crucial to the convergence analysis yet absent in the algorithm.
		(In contrast to \(\lamk*\), the absence of the subscript \(\q\) in \(\gamk*\) emphasizes the independence of the H\"older exponent on the Euclidean stepsize; this notational convention will be adopted throughout.)

		We also remark that the implicit nature of the inclusion \eqref{eq:MM_inclusion} is ubiquitous in algorithms that involve proximal terms of the form $\frac{1}{p}\|x-x^k\|^{p}$ for $p>1$,
		such as the cubic Newton and tensor methods \cite{nesterov2021implementable,doikov2024super,cartis2011adaptive}
		or the high-order proximal point algorithm \cite{luque1984nonlinear,nesterov2021inexact,oikonomidis2023power,laude2023anisotropic}.
		Notice further that the H\"older proximal gradient update for non-Euclidean norms as described above differs from performing a ``scaled'' gradient step followed by a higher-order proximal point step.
		Instead, that corresponds to the anisotropic proximal gradient method \cite{laude2022anisotropic} for choosing $\phi(x)=\frac{H}{1+\q}\|x\|_{1+\q}^{1+\q}$.

		\subsection*{Our contribution}

			Our approach departs from and improves upon existing works in the following aspects.
			\ifarxiv
				\begin{itemize}
			\else
				\begin{itemize}[wide, topsep=0pt, itemsep=0pt]
			\fi
			\item
				Through a novel analysis of \adaPG*^{\rp, r} \cite{latafat2023convergence}, we demonstrate that the class of linesearch-free adaptive methods advanced in
				\ifarxiv
				\cite{malitsky2020adaptive,latafat2023convergence,malitsky2023adaptive,latafat2023adaptive}
				\else
				\cite{malitsky2020adaptive,malitsky2023adaptive,latafat2023convergence,latafat2023adaptive}
				\fi
				are convergent even in the \emph{locally} H\"older differentiable setting, covering in particular the class of semi-algebraic \(C^1\) functions.

			\item
				Our approach bridges the gap between two fundamental approaches to minimizing H\"older-smooth functions:
				it is both \emph{exact} as in \cite{bredies2008forward}, in the sense that it does not involve nor depend on any predefined accuracy, and \emph{universal} akin to the approach in \cite{nesterov2015universal}, for it does not depend on (nor require the knowledge of) problem data such as the H\"older exponent \(\q\).

			\item
				We establish sequential convergence (as opposed to subsequential or approximate cost convergence) with an exact rate
				\begin{equation}
					\min_{k\leq K}\varphi(x^k) - \min \varphi
				\leq
					O(\tfrac1{(K+1)^{\q}}).
				\end{equation}
				 Differently from existing analyses that rely on a global lower bound on the stepsizes to infer convergence and an \(O(\nicefrac1{(K+1)})\) rate in the case \(\q=1\), we identify a scaling of the stepsizes and a lower bound thereof that enables us to tackle  the general \(\q\in(0,1)\) regime.
			\item
				In numerical simulations we show that \adaPG{} performs well on a collection of locally and globally H\"older smooth problems, such as classification with H\"older-smooth SVMs and a $p$-norm version of Lasso. We show that our method performs consistenly better than Nesterov's universal primal gradient method \cite{nesterov2015universal} and in many cases better than its fast variant \cite{nesterov2015universal}, as well as the recently proposed auto-conditioned fast gradient method \cite{li2023simple}.
						\end{itemize}

	\section{Universal, adaptive, without approximation}\label{sec:Universality}

		We consider standard (Euclidean) proximal gradient steps
		\begin{equation}\label{eq:PG}
			x^{k+1}
		=
			\prox_{\gamk*g}(x^k-\gamk*\nabla f(x^k))
		\end{equation}
		for solving \eqref{eq:P} under the following assumptions.

		\begin{assumption}\label{ass:basic}
			The following hold in problem \eqref{eq:P}:
			\begin{enumeratass}[widest=3]
			\item \label{ass:f}%
				\(\func{f}{\R^n}{\R}\) is convex and has locally H\"older continuous gradient of (possibly \emph{unknown}) order \(\q\in(0,1]\).
			\item \label{ass:g}%
				\(\func{g}{\R^n}{\Rinf}\) is proper, lsc, and convex.
			\item
				A solution exists: \(\argmin\varphi\neq\emptyset\).
			\end{enumeratass}
		\end{assumption}

		We build upon a series of adaptive algorithms, starting with a pioneering gradient method in \cite{malitsky2020adaptive} and the follow-up studies
		\ifarxiv
		\cite{latafat2023convergence,malitsky2023adaptive,latafat2023adaptive,zhou2024adabb}
		\else
		\cite{latafat2023convergence,latafat2023adaptive,malitsky2023adaptive,zhou2024adabb}
		\fi
		which contribute with proximal extensions, larger stepsizes, and tighter convergence rate estimates.
		While standard results of proximal algorithms guarantee a descent along the iterates in terms of the cost, distance to solutions, and fixed-point residual individually,
		the key idea behind this class of methods is to eliminate linesearch procedures by \emph{implicitly} ensuring a descent on a (time-varying) combination of the three (see \cref{lemma:main:inequality}).
		This was achieved under local Lipschitz continuity of \(\nabla f\), by exploiting local Lipschitz estimates at consecutive iterates \(x^{k-1}, x^k \in \R^n\) generated by the algorithm such as
		\begin{subequations}\label{eq:lL_noq}%
			\begin{gather}
				\label{eq:l_noq}
				\ell_k
			\coloneqq
				\frac{
					\innprod{x^k-x^{k-1}}{\nabla f(x^k)-\nabla f(x^{k-1})}
				}{
					\|x^k-x^{k-1}\|^{2}
				}
			\shortintertext{and}
				\label{eq:L_noq}
				L_k
			\coloneqq
				\frac{
					\|\nabla f(x^k)-\nabla f(x^{k-1})\|
				}{
					\|x^k-x^{k-1}\|
				}
			\end{gather}
		\end{subequations}
		(with the convention \(\frac{0}{0}=0\)).

		Under the assumption considered in the aforementioned references of local Lipschitz continuity of \(\nabla f\), that is, with \(\q=1\), the estimates \(\ell_k\) and \(L_k\) as in \eqref{eq:lL_noq} remain bounded whenever \(x^k\) and \(x^{k-1}\) range in a bounded set.
		Although this is no more the case in the setting investigated here, for \(\ell_k\) and \(L_k\) may diverge as \(x^k\) and \(x^{k-1}\) get arbitrarily close,
		we show that these \emph{Lipschitz} estimates can still be employed even if \(\nabla f\) is merely locally \emph{H\"older} continuous.

		Specifically, we will consider the following stepsize update rule
		\begin{align}
			\textstyle
			\gamk*
		=
			\gamk\min\Biggl\{{}
		&
			\sqrt{\frac{1}{\rp}+\frac{\gamk}{\gamma_{k-1}}},
		\label{eq:adaPG_noq:body}
		\ifarxiv~\else
			\\
			\nonumber
			\myampersand
		\fi
				\sqrt{
					\frac{1}{
						2\left[\gamk^2L_k^2-(2-\rp)\gamk\ell_k -(\rp-1)\right]_+
					}
				}
			\Biggr\}
		\end{align}
		for some \(\rp\in[1,2]\), where \([{}\cdot{}]_+\coloneqq\max\set{{}\cdot{},0}\).
		This correponds to the update rule of the algorithm \adaPG^{\rp,r} of \cite{latafat2023convergence} specialized to the choice \(r=\nicefrac{\rp}{2}\) for the second parameter.
		This restriction is nevertheless general enough to recover the update rules of \cite[Alg. 1]{malitsky2020adaptive}, \cite[Alg. 2.1]{latafat2023adaptive}, and \cite[Alg. 1]{malitsky2023adaptive} (\(\rp=1\)), as well as the one of \cite[Alg. 3]{malitsky2023adaptive} (\(\rp=\nicefrac{3}{2}\)), see \cite[Rem. 2.4]{latafat2023adaptive}.
		In particular, our analysis in the generality of \(\rp\in[1,2]\) demonstrates that \emph{all these adaptive algorithms are convergent in the locally H\"older (convex) setting}.

		A crucial challenge in the locally  H\"older setting is the lack  of a positive uniform lower bound for the stepsize sequence \(\seq{\gamk}\) generated by \eqref{eq:adaPG_noq:body}.
		To mitigate this, we{}  factorize \(\gamk\) as{}  		\begin{equation}\label{eq:lamk}
			\gamk  		\mathrel{{=}}
			\lamk \|x^k-x^{k-1}\|^{1-\q},
		\end{equation}
		introducing scaled stepsizes \(\lamk\) as suggested by the upper bound minimization procedure \eqref{eq:MM_inclusion}.
		This allows us to
				normalize{}  the{}  		H\"older inequalities into Lipschitz-like ones, see \cref{thm:Hk}.
		(Throughout, the subscript \(\q\) shall be used for quantities with dependence on \(\q\) for clarity of exposition.)
		Our analysis relies on showing that this scaled stepsize is lower bounded whenever the second term in \eqref{eq:adaPG_noq:body} is active (see \cref{thm:lammin}).
		We emphasize that this quantity, while crucial in our convergence analysis, does not appear in the algorithm, which only uses the estimates \eqref{eq:lL_noq} and the update rule \eqref{eq:adaPG_noq:body}, neither of which depend on (the knowledge of) the local H\"older order \(\q\).

		\subsection{H\"older continuity estimates}\label{sec:Holder}%

			In this section, we set up some basic facts about H\"older continuity of \(\nabla f\) that will be essential in our analysis.
			We again emphasize that our convergence analysis makes mere use of \emph{existence} of \(\q\in (0, 1]\) (see \cref{ass:f}), but the algorithm is independent of (the knowledge of) this exponent, cf. \adaPG{} (\cref{alg:adaPG}).

			Local H\"older continuity of \(\nabla f\) of order \(\q\) (which we shall refer to as local \(\q\)-H\"older continuity of \(\nabla f\) for brevity) amounts to the existence, for every convex and bounded set \(\Omega\subset\R^n\), of a constant \(\L_\Omega>0\) such that
			\[
				\|\nabla f(y)-\nabla f(x)\|\leq\L_\Omega\|y-x\|^{\q}
			\quad
				\forall x,y\in\Omega.
			\]
			Note that both norms are the standard Euclidean 2-norms.
			We remark that the limiting case \(\q=0\) amounts to \emph{sub}\-gra\-di\-ents of \(f\) being bounded on bounded sets, that is, to \(f\) being merely convex and real valued with no differentiability requirements.
			Although not covered by our convergence results in \cref{sec:convergence}, the preliminary lemmas collected in \cref{sec:lemmas} still remain valid for any real-valued convex \(f\).
			To clearly emphasize this fact, whenever applicable we shall henceforth specify ``(possibly with \(\q=0\))'' when invoking \cref{ass:basic};
			in this case, the notation \(\nabla f\) shall indicate any \emph{sub}gradient map \(\func{\nabla f}{\R^n}{\R^n}\) with \(\nabla f(x)\in\partial f(x)\), and
			\begin{equation}
				\def\q{0}
				\L_\Omega\leq2\lip_\Omega f
			\end{equation}
			is bounded by twice the Lipschitz modulus for \(f\) on \(\Omega\) \cite[Thm. 24.7]{rockafellar1970convex}.

			Throughout, we will make use of the following inequalities, which reduce to well-known Lipschitz and cocoercivity properties of \(\nabla f\) when \(\q=1\).
			The proof of the second assertion can be found in \cite[Lem. 1]{yashtini2016global}; our cocoercivity-like claims are a slight refinement of known global and/or scalar versions, see \eg,
			\cite[Cor. 18.14]{bauschke2017convex} and \cite[Prop.~1]{ying2017unregularized}.
			\ifarxiv
			\begin{WhereToPutThis}
				The simple proofs of this and the following lemma are provided in the dedicated \cref{proofsec:Holder}.
			\end{WhereToPutThis}
			\fi

			\begin{fact}[H\"older-smoothness inequalities]\label{thm:Lq}%
				Suppose that \(\func{h}{\R^n}{\R}\) is convex and \(\nabla{ h}\) is \(\q\)-H\"older continuous with modulus \(\L_{E}>0\) { on a convex set \(E\subseteq\R^n\)} for some \(\q\in[0,1]\).
				Then, for every \(x,y\in{ E}\) the following hold:
				\begin{enumerate}[widest=4]
				\item \label{thm:Lq:innprod:leq}%
					\(
						\innprod{x-y}{\nabla{{ h}}(x)-\nabla{ h}(y)}
					\leq
						\L_{ E}\|x-y\|^{1+\q}
					\)%
				\item \label{thm:Lq:f:leq}%
					\(
						{ h}(y)-{ h}(x)-\innprod{\nabla{ h}(x)}{y-x}
					\leq
						\tfrac{\L_{ E}}{1+\q}\|x-y\|^{1+\q}
					\).%
				\end{enumerate}
				If \(\q\neq0\) and \(\nabla h\) is \(\q\)-H\"older continuous on an enlarged set \(\overline E\coloneqq E+\cball{0}{\diam E}\) with modulus \(\L_{\overline{E}}\), then the following local cocoercivity-type estimates also hold:
				\begin{enumerate}[resume]
				\item \label{thm:Lq:innprod:geq}%
					\(\ifarxiv\else\mathtight\fi
						\frac{2\q}{1+\q}
						\frac{1}{\L_{\overline{E}}^{1\!/\!\q}}
						\|\nabla h(x)-\nabla h(y)\|^{\frac{1+\q}{\q}}
					\ifarxiv{}\leq{}\else\!\leq\!\fi
						\mathrlap{\innprod{x-y}{\nabla{ h}(x)-\nabla{ h}(y)}}
					\)%
				\item \label{thm:Lq:f:geq}%
					\(\ifarxiv\else\mathtight\fi
						\frac{\q}{1+\q}
						\frac{1}{\L_{\overline E}^{1\!/\!\q}}
						\|\nabla h(x)-\nabla h(y)\|^{\frac{1+\q}{\q}}
					\leq
						h(y)-h(x)-
						\langle\nabla h(x),\allowbreak y-x\rangle
					\).%
				\end{enumerate}
			\end{fact}

			\begin{algorithm*}[t!]
				\caption[]{\AdaPG*: A universal adaptive proximal gradient method}
				\label{alg:adaPG}

		\begin{algorithmic}[1]
		\itemsep=3pt
		\Require
			starting point \(x^{-1}\in\R^n\),~
			stepsizes \(\gamma_0\geq\gamma_{-1}>0\),~
			parameter \(\rp\in [1,2]\)

		\item[\algfont{Initialize:}]
			\(x^0=\prox_{\gamma_0g}(x^{-1}-\gamma_0\nabla f(x^{-1}))\)

		\item[\algfont{Repeat for} \(k=0,1,\ldots\) until convergence]

		\State \label{state:PG:gamk*}%
			With \(\ell_k, L_k\) given in \eqref{eq:lL_noq}, define the stepsize as \hfill {\gray \scriptsize (with notation \([z]_+ = \max\set{z,0}\) and convention \(\tfrac10= \infty\))}
			\begin{equation}\label{eq:gamk*}
				\gamk*
			=
				\gamk\min\set{
					\sqrt{\frac1{\rp}+\frac{\gamk}{\gamma_{k-1}}},\,
					\frac{
						1
					}{
						\sqrt{2\left[\gamk^2L_k^2-(2-\rp)\gamk\ell_k + 1-\rp\right]_+}
					}
				}
			\end{equation}

		\State
			\(x^{k+1}=\prox_{\gamk*g}(x^k-\gamk*\nabla f(x^k))\)
		\end{algorithmic}
			\end{algorithm*}

			Based on the inequalities in \cref{thm:Lq}, given a sequence \(\seq{x^k}\) we define local estimates of H\"older continuity of \(\nabla f\) with \(\q\in [0,1]\) as follows:
			\begin{subequations}\label{eq:lL}%
				\begin{align}
					\label{eq:l}
					\lk
				\coloneqq{} &
					\frac{
						\innprod{x^k-x^{k-1}}{\nabla f(x^k)-\nabla f(x^{k-1})}
					}{
						\|x^k-x^{k-1}\|^{1+\q}
					}
				\shortintertext{and}
					\label{eq:L}
					\Lk
				\coloneqq{} &
					\frac{
						\|\nabla f(x^k)-\nabla f(x^{k-1})\|
					}{
						\|x^k-x^{k-1}\|^{\q}
					}.
				\end{align}
			\end{subequations}
			Let us draw some comments on these quantities.
			Considering the scaled stepsize \(\lamk\) given in \eqref{eq:lamk}, it is of immediate verification that
			\begin{equation}\label{eq:lam:gam}
				\lamk \Lk = \gamk L_k,
			\quad
				\lamk \lk = \gamk \ell_k.
			\end{equation}
			Moreover, observe that
			\begin{equation}\label{eq:lLL}
				\lk\leq\Lk\leq\L_\Omega
			\end{equation}
			holds whenever \(\L_\Omega\) is a \(\q\)-H\"older modulus for \(\nabla f\) on a compact convex set \(\Omega\) that contains both \(x^{k-1}\) and \(x^k\), the first inequality following from a simple application of Cauchy-Schwartz.
			We also remark that defining \(\lk\) and \(\Lk\) as above in place of a cocoercivity-like estimate
			\[
				\ck
			\coloneqq
				\biggl(
					\frac{2\q}{1+\q}
					\frac{
						\|\nabla f(x^k)-\nabla f(x^{k-1})\|^{1+\frac{1}{\q}}
					}{
						\innprod{x^k-x^{k-1}}{\nabla f(x^k)-\nabla f(x^{k-1})}
					}
				\biggr)^{\q}
			\]
			causes no loss of generality, since each one among \(\ck\), \(\lk\) and \(\Lk\) can be derived based on the other two.
			The use of \(\lk\) and \(\Lk\) provides nevertheless a simpler and more straightforward H\"older estimate, contrary to \(\ck\) which instead involves counterintuitive powers and coefficients, as well as a potentially looser H\"older modulus, and fails to cover the limiting case \(\q=0\), cf. \cref{thm:Lq:innprod:geq}.

			Let us denote the forward operator by
			\begin{equation}\label{eq:Hk}
				\Hk\coloneqq\id-\gamk\nabla f.
			\end{equation}
			The subgradient characterization of the proximal gradient update \eqref{eq:PG} then reads
			\begin{equation}\label{eq:nablaPhi}
				\Hk(x^{k-1}) - \Hk(x^{k}) \in  \gamk  \partial \varphi(x^k).
			\end{equation}
			As in \cite{latafat2023adaptive}, the combined use of \(\lk\) and \(\Lk\) yields a local H\"older modulus for the forward operator, though in this work it will be convenient to express it with respect to the scaled stepsize \(\lamk\) as in \eqref{eq:lamk}.

			\begin{lemma}\label{thm:Hk}%
				Let \(\lk\) and \(\Lk\) be as in \eqref{eq:lL} for some \(x^{k-1},x^k\in\R^n\), and let \(\Hk\) be as in \eqref{eq:Hk} for some \(\gamk>0\).
				Then, for any \(\q\in[0,1]\) and with \(\lamk\) as in \eqref{eq:lamk} it holds that
				\begin{align*}
					M_k^2
				\coloneqq{} &
					\tfrac{
						\|\Hk(x^k)-\Hk(x^{k-1})\|^2
					}{
						\|x^k-x^{k-1}\|^2
					}
				=
					1+\lamk^2\Lk^2-2\lamk\lk
				\\
				={} &
					1+\gamk^2L_k^2-2\gamk\ell_k.
				\end{align*}
			\end{lemma}

			Being \(M_k\) a \emph{Lipschitz} estimate, unless \(\nabla f\) is locally Lipschitz continuous there is no guarantee that \(M_k\) is bounded for pairs \(x^{k-1},x^k\) ranging in a compact set.
			The \(\q\)-H\"older estimate of the forward operator \(\Hk\), which instead \emph{is} guaranteed to be bounded on bounded sets (for bounded \(\seq{\gamk}\)), is given by
			\begin{align*}
			\nonumber
				\Mk
			\coloneqq{} &
				M_k\|x^k-x^{k-1}\|^{1-\q}
			=
				\tfrac{
					\|\Hk(x^k)-\Hk(x^{k-1})\|
				}{
					\|x^k-x^{k-1}\|^{\q}
				}
			\\
			\leq{} &
				(1 + \lamk\Lk)\|x^k-x^{k-1}\|^{1-\q},
			\numberthis \label{eq:hkq:bound}
			\end{align*}
			where the inequality follows from the triangle inequality and the definition of \(\Lk\), cf. \eqref{eq:L}.

	\section{\texorpdfstring{\protect{\AdaPG*}}{AdaPG} revisited}\label{sec:algorithm}%

		In this section \adaPG\ is presented in \cref{alg:adaPG} for solving composite problems \eqref{eq:P}. The main oracles of the algorithm are plain proximal and gradient evaluations. We refer to \cite[\S6]{beck2017first} for examples of functions with easy to evaluate proximal maps.

		\AdaPG\ incorporates the simple stepsize update rule \eqref{eq:gamk*} with a parameter \(\rp \in [1,2]\) that strikes a balance between speed of recovery from small values (e.g., due to steep or ill-conditioned regions), and magnitude of the stepsize dictated by the second term.
		If \(\rp =1\), whenever \(\gamk^2 L_k^2 \leq \gamk \ell_k\), the second term reduces to \(\nicefrac10=\infty\), and \(\gamk* = \gamk \sqrt{1+\nicefrac{\gamk}{\gamma_{k-1}}}\) strictly increases.
		On the other end of the spectrum, if \(\rp = 2\) the update reduces to
		\[
			\gamk*
		=
			\gamk\min\set{
				\sqrt{\tfrac12+\tfrac{\gamk}{\gamma_{k-1}}},\,
				\tfrac{
					1
				}{
					\sqrt{2\left[\gamk^2L_k^2 - 1\right]_+}
				}
			},
		\]
		where having the first term active (for instance if \(\gamk L_k \leq 1\)) for two consecutive updates  already ensures an increase in the stepsize owing to the simple observation that \({\nicefrac12 + \sqrt{\nicefrac12}}>1\).

		While \adaPG{} has the ability to recover from a potentially bad choice of initial stepsizes \(\gamma_{0}, \gamma_{-1}\), the behavior of the algorithm during the first iterations can be impacted negatively.
		To eliminate such scenarios, \(\gam_0\) can be refined by running offline proximal gradient updates. Specifically, starting from the initial point \(x^{-1}\), \(\gam_0\) can be updated as the inverse of either one of \eqref{eq:L} or \eqref{eq:l} evaluated between \(x^{-1}\) and the obtained point. If the updated stepsize is substantially smaller than the original one, the same procedure may be repeated an additional time.
		Once a reasonable \(\gam_0\) is obtained, we suggest selecting \(\gamma_{-1}\) small enough such that
		\(
			\sqrt{\frac{1}{\rp}+\frac{\gam_0}{\gam_{-1}}}
		\geq
			\frac{1}{\sqrt{2L_0}},
		\)
		ensuring that \(\gam_1\) would be proportional to the inverse of \(L_0\).
		We remark that the choice of \(\gam_{-1}\) does not affect the sequential convergence results of \cref{thm:convergence}.
		It does nevertheless affect the constant in our sublinear rate results of \cref{thm:sublinear}, through possibly having a larger \(\rho_{\rm max}\) therein, although this effect is a mere theoretical technicality with negligible practical implications.

		\subsection{Preliminary lemmas}\label{sec:lemmas}%

			This subsection collects some preliminary results adaptated from
			\ifarxiv
			\cite{malitsky2020adaptive,latafat2023convergence,malitsky2023adaptive,latafat2023adaptive}
			\else
			\cite{malitsky2020adaptive,malitsky2023adaptive,latafat2023convergence,latafat2023adaptive}
			\fi
			that hold true under convexity assumption without further restrictions.
			\ifarxiv
			\begin{WhereToPutThis}
				All the proofs, although differing only by negligible adaptations, are nevertheless provided in the dedicated \cref{proofsec:lemmas} to demonstrate their independence of \(\q\in[0,1]\).
			\end{WhereToPutThis}
			\fi
			In particular, the next lemma can be viewed as a counterpart of the well known \emph{firm} nonexpansiveness (FNE) of \(\prox_{\gamma g}\), which is recovered when \(\gamk*=\gamk\), and offers a refinement of the nonexpansiveness-like inequality in \cite[Lem. 12]{malitsky2023adaptive} that follows after an application of Cauchy-Schwarz.

			\begin{lemma}[FNE-like inequality]\label{thm:FNE}%
				Let \cref{ass:g} hold and \(f\) be differentiable.
				For any \(\gamk>0\) and denoting \(\rhok*\coloneqq\nicefrac{\gamk*}{\gamk}\), iterates \eqref{eq:PG} satisfy the following:
				\begin{align*}
					\tfrac1{\rhok*}\|x^{k+1}-x^k\|^2
				\leq{} &
					\innprod{\Hk(x^{k-1})-\Hk(x^k)}{x^k-x^{k+1}}
				\\
				\leq{} &
					\rho_{k+1}\|\Hk(x^{k-1})-\Hk(x^k)\|^2.
				\end{align*}
			\end{lemma}

			By combining this inequality with the identity in \cref{thm:Hk} we obtain the following.

			\begin{corollary}\label{thm:resratio}%
				Let \cref{ass:f,ass:g} hold (possibly with \(\q=0\)).
				For any \(\gamk>0\), and with \(M_k\) and \(\lamk\) as in \eqref{eq:Hk} and \eqref{eq:lamk}, iterates \eqref{eq:PG} satisfy
				\begin{equation}
					\|x^{k+1}-x^k\|
				\leq
					\tfrac{\gamk*}{\gamk}
					M_k
					\|x^{k-1}-x^k\|.
				\end{equation}
			\end{corollary}

			We next present the main descent inequality taken from \cite[Lem. 2.2]{latafat2023adaptive}.
			Its proof is nevertheless included in the appendix to demonstrate its independence of \(\q\):
			it relies on mere use of convexity inequalities and \emph{identities} involving \(L_k, \ell_k\) as in \cref{thm:FNE,thm:Hk} to express norms and inner products in terms of \(\|x^k-x^{k-1}\|^2\).
			It reveals that for any \(\rp\geq0\) and \(x^\star\in\argmin\varphi\), up to a proper stepsize selection, the function
			\begin{align}
				\nonumber
					\U_k(x^\star)
				\coloneqq{} &
					\tfrac12\|x^k-x^\star\|^2
					+
					\tfrac12\|x^k-x^{k-1}\|^2
				\\
				&
					+
					\gamk(1+\rp\rhok)P_{k-1}
				\label{eq:Uk}
			\end{align}
			monotonically decreases, where we introduced the symbol
			\begin{equation}\label{eq:Pk}
				P_k
			\coloneqq
				\varphi(x^k)-\min\varphi
			\end{equation}
			for the sake of conciseness.

			\begin{lemma}[main inequality]\label{lemma:main:inequality}%
				Let \cref{ass:basic} hold (possibly with \(\q=0\)), and consider a sequence generated by proximal gradient iterations \eqref{eq:PG} with \(\gamk>0\) and \(\rhok*\coloneqq\nicefrac{\gamk*}{\gamk}\).
				Then, for any \(\rp\geq0\), \(x^\star\in\argmin\varphi\), and \(k\in \N\) the following holds:
				\begin{align}
				\nonumber
					\U_{k+1}(x^\star)
				\leq{} &
					\U_k(x^\star)
					-
					\gamk\bigl(1+\rp \rhok- \rp\rhok*^2\bigr)
					P_{k-1}
				\\
			\ifarxiv\else
				\nonumber
			\fi
				&
					-
					\bigl\{
						\tfrac{1}{2}
						-
						\rhok*^2\bigl[
							\gamk^2L_k^2
							-
							\gamk\ell_k(2-\rp)
			\ifarxiv\else
				\\
				\myampersand
					\hphantom{-\bigl\{~}
			\fi
							+
							1-\rp
						\bigr]
					\bigr\}
					\|x^k-x^{k-1}\|^2.
				\label{eq:SD}
				\end{align}
			\end{lemma}

			Apparently, the stepsize update rule of \adaPG{} is designed so as to make the coefficients of \(P_{k-1}\) and \(\|x^k-x^{k-1}\|^2\) on the right-hand side of \eqref{eq:SD} negative, so that the corresponding proximal gradient iterates monotonically decrease the value of \(\U_k(x^\star)\).

			\begin{lemma}[basic properties of \adaPG]\label{thm:Lyapunov}%
				Under \cref{ass:basic} (possibly with \(\q=0\)), the following hold for the iterates generated by \adaPG:
				\begin{enumerate}
				\item \label{thm:descent}%
					\(\seq{\U_k(x^\star)}\) as defined in \eqref{eq:Uk} decreases and converges to a finite value.

				\item \label{thm:bounded}%
					The sequence \(\seq{x^k}\) is bounded and admits at most one optimal limit point.

				\item \label{thm:sumgamk}%
					\(P_K^{\rm min}\leq\U_0(x^\star)\big/\bigl(\sum_{k=1}^{K+1}\gamk\bigr)\) for every \(K\geq1\),
					where
					\(P_k^{\rm min}\coloneqq\min_{0\leq i\leq k}P_i\).
				\end{enumerate}
			\end{lemma}

			\begin{remark}
				The validity of \cref{thm:sumgamk} also when \(\q=0\) hints that having \(\sum_{k\in\N}\gamk=\infty\) suffices to infer convergence results for the proximal \emph{sub}gradient method without differentiability of \(f\).
				Whether, or under which conditions, this is really the case is currently an open problem.
			\end{remark}

		\subsection{Convergence and rates}\label{sec:convergence}%

			Distinguishing between the iterates in which the stepsize is updated according to the first or the second element in the minimum of \eqref{eq:gamk*} will play a fundamental role in our analysis.
			For this reason, it is convenient to introduce the following notation:
			\begin{subequations}\label{eq:K12}
				\begin{align}
				\label{eq:K1}
					K_1
				\coloneqq{} &
					\set{k\in\N}[
						\gamk*=\gamk\sqrt{\tfrac{1}{\rp}+\tfrac{\gamk}{\gam_{k-1}}}
					]
				\shortintertext{and}
				\label{eq:K2}
					K_2
				\coloneqq{} &
					\N\setminus K_1.
				\end{align}
			\end{subequations}
			Unlike its locally Lipschitz counterpart (\(\q=1\)), in the H\"ol\-der setting, a global lower bound for the stepsize sequence \(\seq{\gamk}\) cannot be expected.
			Nevertheless, a lower bound for the \emph{scaled} stepsizes \(\lamk\) whenever \(k \in K_2\) is sufficient to ensure convergence.

			\begin{lemma}\label{thm:lammin}%
				Let \cref{ass:basic} hold (possibly with \(\q=0\)), and consider the iterates generated by \adaPG.
				Then, with \(K_2\) as in \eqref{eq:K2}, for every \(k\in K_2\)%
				\begin{equation}\label{eq:lamk*:Dk:0}
					\lamk
				\geq
					\frac{1}{\sqrt2\Lk\rho_{\rm max}}
				~~\text{and}~~
					\rhok*
				\geq
					\tfrac{1}{\sqrt{2}\lamk\Lk},
				\end{equation}
				where
				\(
					\rho_{\rm max}
				\coloneqq
					\max\set{
						\frac{1}{2}\bigl(1+\sqrt{1+\nicefrac{4}{\rp}}\bigr),
						\frac{\gam_0}{\gam_{-1}}
					}
				\).
			\end{lemma}
			\begin{WhereToPutThis}
			\begin{proof}

					Owing to \(\rhok* \leq \sqrt{\nicefrac1{\rp}+ \rhok}\) as ensured in \eqref{state:PG:gamk*}, it can be verified with a trivial induction argument that
		\begin{equation}\label{eq:rhomax}
			\rhok
		\leq
			\rho_{\rm max}
		\coloneqq
			\max\set{\tfrac{1}{2}(1+\sqrt{1+\nicefrac{4}{\rp}}), \rho_0}
		\quad
			\text{for all } k \geq 0.
		\end{equation}

		If \(k\in K_2\), then \(\gamk*\) coincides with the second update in \eqref{eq:gamk*}, and thus
		\begin{equation}\label{eq:2active:lb}
			\rho_{\rm max}
		\geq
			\rhok*
		=
			\frac{1}{
				\sqrt{2\left[\lamk^2\Lk^2-(2-\rp)\lamk\lk + 1- \rp\right]_+}
			}
		\geq
			\frac{1}{\sqrt2\lamk\Lk},
		\end{equation}
		completing the proof.
			\end{proof}
			\end{WhereToPutThis}

			The anticipated lower bound on \(\seq{\lamk}[k\in K_2]\) will follow from \eqref{eq:lLL}, once boundedness of the sequence \(\seq{x^k}\) generated by \adaPG{} is established.

			\begin{WhereToPutThis}

				In our convergence analysis we will need the following lemma that extends \cite[Lem. B.2]{latafat2023convergence} by allowing a vanishing stepsize.
	As a result it is only \(\gamk*\) times the cost that can be ensured to converge to zero, which will nevertheless prove sufficient for our convergence analysis in the proof of \cref{thm:convergence}.

	\begin{lemma}\label{thm:Pkto0}%
		Suppose that a sequence \(\seq{x^k}\) converges to an optimal point \(x^\star\in\argmin\varphi\), and for every \(k\) let \(\bar x^k\coloneqq\prox_{\gamk* g}(x^k-\gamk*\nabla f(x^k))\) with \(\seq{\gamk}\subset\R_{++}\) bounded.
		Then, \(\seq{\bar x^k}\) too converges to \(x^\star\) and \(\seq{\gamk*(\varphi(\bar x^k) - \min \varphi)}\to 0\).
	\end{lemma}
	\begin{proof}
		By nonexpansiveness of the proximal mapping
		\[
			\|\bar x^k - x^\star\|
		\leq
			\|x^k - x^\star - \gamk*(\nabla f(x^k) - \nabla f(x^\star))\|
		\leq
			\|x^k - x^\star\|
			+
			\gamk*\|\nabla f(x^k) - \nabla f(x^\star)\| \to 0,
		\]
		where we used the fact that \(x^\star = \prox_{\gamk* g}(x^\star - \gamk* \nabla f(x^\star))\) for any \(\gamk*>0\) in the first inequality, and boundedness of \(\gamk*\) in the last implication.
		Moreover, for every \(k\in\N\) one has
		\[
			\gamk*(\varphi(\bar x^k) - \min \varphi)
		=
			\gamk* (f(\bar x^k)+g(\bar x^k) - \min \varphi)
		\leq
			\gamk*(f(\bar x^k)-f(x^\star))
			-
			\langle x^k - \gamk* \nabla f(x^k) - \bar x^k, x^\star - \bar x^k \rangle,
		\]
		where in the inequality we used the subgradient characterization of the proximal mapping.
		The inner product vanishes since both \(x^k\) and \(\bar x^k\) converge to \(x^\star\), and the claim follows by  continuity of \(f\) and lower semicontinuity of \(\varphi\).
	\end{proof}
			\end{WhereToPutThis}

			\begin{theorem}[convergence]\label{thm:convergence}%
				Under \cref{ass:basic}, the sequence \(\seq{x^k}\) generated by \adaPG\ converges to some \(x^\star \in \argmin \varphi\).
			\end{theorem}
			\begin{WhereToPutThis}
			\begin{proof}

		We first show two intermediate claims.
		\def\currentlabel{thm:convergence}%
		\begin{claims}
		\item \label{claim:infP}%
			{\em If \(\q>0\), then \(\inf_{k\in\N}P_k=0\), and in particular \(\seq{x^k}\) admits a (unique) optimal limit point.}

			If \(\sup_{k\in\N}\gamk=\infty\), then we know from \cref{thm:sumgamk} that \(\liminf_{k\to\infty}P_k=0\).
			Suppose instead that \(\seq{\gamk}\) is bounded.
			Then, the set \(K_2\) as in \eqref{eq:K2} must be infinite.
			Let \(\L_\Omega\) be a \(\q\)-H\"older modulus for \(\nabla f\) on a compact convex set \(\Omega\) that contains all the iterates \(x^k\), ensured to exist by \cref{thm:bounded}.
			Since \(\Lk\leq\L_{\Omega}\), it follows from \cref{thm:lammin} that
			\begin{equation}\label{eq:lammin}
				\lamk
			\geq
				\lam_{\rm min}
			\coloneqq
				\frac{1}{\sqrt2\L_{\Omega}\rho_{\rm max}}
			\quad
				\forall k\in K_2,
			\end{equation}
			hence from \eqref{eq:sumPk} that
			\[
				\sum_{k\in K_2}\|x^k-x^{k-1}\|^{1-\q}(1+\rp\rhok-\rp\rhok*^2)P_{k-1}
			<
				\infty.
			\]
			Noticing that \(1+\rp\rhok-\rp\rhok*^2=0\) for \(k\notin K_2\), necessarily \(1+\rp\rhok-\rp\rhok*^2\not\to0\) as \(K_2\ni k\to\infty\) (or, equivalently, as \(k\to\infty\)), for otherwise \(\liminf_{k\to \infty} \rho_k >1\) and thus \(\gamk\nearrow\infty\).
			Therefore, there exists an infinite set \(\tilde K_2\subseteq K_2\) such that \(1+\rp\rhok-\rp\rhok*^2\geq\varepsilon>0\) for all \(k\in\tilde K_2\), implying that
			\[
				\sum_{k\in\tilde K_2}\|x^k-x^{k-1}\|^{1-\q}P_{k-1}
			<
				\infty.
			\]
			Thus, \(\lim_{\tilde K_2\ni k\to\infty}\|x^k-x^{k-1}\|=0\) (or \(\liminf_{k\in\tilde K_2}P_{k-1}=0\), in which case there is nothing to show).
			For any \(x^\star\in\argmin\varphi\) we thus have
			\begin{align*}
				0
			\leq
				P_k
			=
				\varphi(x^k)-\min\varphi
			\leq{} &
				\innprod{x^k-x^\star}{\tfrac{x^{k-1}-x^k}{\gamk}-\bigl(\nabla f(x^{k-1})-\nabla f(x^k)\bigr)}
			\\
			\numberthis\label{eq:Pkbound}
			\leq{} &
				\|x^k-x^\star\|
				\left(
					\tfrac{1}{\gamk}
					\|x^{k-1}-x^k\|
					+
					\|\nabla f(x^{k-1})-\nabla f(x^k)\|
				\right)
			\\
			={} &
				\|x^k-x^\star\|
				\left(
					\tfrac{1}{\lamk\|x^{k-1}-x^k\|^{1-\q}}
					\|x^{k-1}-x^k\|
					+
					\Lk
					\|x^{k-1}-x^k\|^{\q}
				\right)
			\\
			\leq{} &
				\|x^k-x^\star\|
				\|x^{k-1}-x^k\|^{\q}
				\left(
					\tfrac{1}{\lam_{\rm min}}
					+
					\L_\Omega
				\right)
			\quad
				\forall k\in\tilde K_2.
			\end{align*}
			Since \(\q>0\), by taking the limit as \(\tilde K_2\ni k\to\infty\) we obtain that \(\lim_{\tilde K_2\ni k\to\infty}P_k=0\).

		\item
			{\em If \(\q>0\) and \(\seq{\gamk}\) is bounded, then \(\seq{x^k}\) converges to a solution.}

			Suppose first that \(\seq{\gamk}\) is bounded.
			Consider a subsequence \(\seq{x^k}[k\in K]\) such that \(\lim_{K\ni k\to\infty}P_k=0\), which exists and converges to a solution \(x^\star\) by \cref{claim:infP}.
			Since \(\seq{\gamk}\) is bounded, in light of \cref{thm:Pkto0} also \(x^{k+1}\to x^\star\) and, in turn, \(x^{k+2}\to x^\star\) as well.
			Then,
			\[
				\U_{k+1}(x^\star)
			=
				\tfrac{1}2\|x^{k+1}-x^\star\|^2
				+
				\tfrac{1}2\|x^{k+1}-x^k\|^2
				+
				\gamk\left(1+\rp\rhok*\right)P_k
			\to
				0
			\quad
				\text{as }K\ni k\to\infty,
			\]
			and thus \(\frac{1}2\|x^k-x^\star\|\leq\U_k(x^\star)\to0\) as \(k\to\infty\), since the entire sequence \(\seq{\U_k(x^\star)}\) is convergent.
		\end{claims}

		To conclude the proof ot the theorem, it remains to show that also in case \(\seq{\gamk}\) is unbounded the sequence \(\seq{x^k}\) converges to a solution.
		To this end, let us suppose now that \(\gamk\) is not bounded.
		This case requires requires a few more technical steps, which can nevertheless almost verbatim be adapted from the proof of \cite[Thm. 2.4(ii)]{latafat2023convergence}. See also \cite[Thm. 2.3(iii)]{latafat2023adaptive} for an alternative argument;
		we emphasize that the difference with both aforementioned works is that the stepsize sequence is not guaranteed to be bounded away from zero.

		We start by observing that \cref{claim:infP} and \cref{thm:bounded} ensure that an optimal limit point \(x^\star\in\argmin\varphi\) exists.
		It then suffices to show that \(\U_k(x^\star)\) converges to zero.
		To arrive to a contradiction, suppose that this is not the case, that is, that
		\(
			U \coloneqq \lim_{k \to \infty}\U_k(x^\star)>0
		\).
		We shall henceforth proceed by intermediate claims that follow from this condition, eventually arriving to a contradictory conclusion.

		\begin{claims*}
		\item \label{claim*:PG:gamk*}%
			{\em For any \(K\subseteq\N\),
				\(
					\lim_{K\ni k\to\infty} x^k = x^\star
				\)
				holds iff
				\(
					\lim_{K\ni k\to\infty} \gamk* = \infty
				\).
			}%

			The implication ``\(\Leftarrow\)'' follows from
			\[
				\gamk P_{k-1} \leq \U_k(x^\star) < \infty
			\]
			since \(\seq{x^k}\) is bounded and \(x^\star\) is its unique optimal limit point.

			Suppose now that \(\seq{x^k}[k\in K]\to x^\star\).
			To arrive to a contradiction, up to possibly extracting another subsequence suppose that \(\seq{\gamk*}[k\in K]\to\bar\gamma \in [0, \infty)\).
			Then, it follows from \cref{thm:Pkto0} that \(\seq{x^{k+1}}[k\in K]\to x^\star\) and \(\seq{\gamk*P_{k+1}}[k\in K]\to0\).
			As shown in \eqref{eq:rhomax}
			\begin{equation}\label{eq:rhomax:2}
				\rhok \leq \rho_{\rm max} \quad \text{for all } k \geq 0
				\end{equation}
			which in turn implies \(\seq{\gam_{k+2} P_{k+1}}[k\in K] \to 0\) and that \(\seq{\gamma_{k+2}}[k\in K]\) is also bounded, we may iterate and infer that also \(\seq{x^{k+2}}[k\in K]\) converges to \(x^\star\).
			Recalling the definition of \(\U_k\) in \eqref{eq:Uk},
			\begin{align*}
				\U_{k+2}(x^\star)
			\coloneqq{} &
				\tfrac12\|x^{k+2}-x^\star\|^2
				+
				\tfrac12\|x^{k+2}-x^{k+1}\|^2
				+
				\gam_{k+2}(1+\rp\rho_{k+2})P_{k+1}
			\to 0,
			\end{align*}
			contradicting \(U=\lim_{K\ni k\to\infty}\U_{k+2}(x^\star)>0\).

		\item
			{\em Suppose that \(\seq{x^k}[k\in K]\to x^\star\); then also \(\seq{x^{k-1}}[k\in K]\to x_\star\).}

			It follows from the previous claim that \(\lim_{K\ni k\to\infty}\gamk*=\infty\).
			Because of \eqref{eq:rhomax:2}, one must also have \(\lim_{K\ni k\to\infty}\gamk=\infty\).
			Invoking again the previous claim, by the arbitrarity of the index set \(K\) the assertion follows.

		\item \label{claim*:gamkLk:infty}
			{\em
				Suppose that \(\seq{x^k}[k\in K]\to x^\star\); then \(\seq{\gam_{k-1} L_{k-1}}[k\in K]\to\infty\){}  and \(\seq{\rhok}[k\in K]\to 0\).
			}%

			Using the previous claim twice, \(x^{k-1},x^{k-2}\to x^\star\) as \(K\ni k\to\infty\).
			In particular
			\begin{equation}\label{eq:shifted:res}
				\lim_{K\ni k\to\infty}\|x^{k-1}-x^{k-2}\|^2=0.
			\end{equation}
			From the expression \eqref{eq:Uk} of \(\U_k\)  we then have
			\begin{equation}\label{eq:gamkPk:U}
				\lim_{K\ni k\to\infty}\gamk(1+\rp\rhok)P_{k-1}
			=
				U,
			\end{equation}
			where we remind that by contradiction assumption \(U\coloneqq\lim_{k\to\infty}\U_k(x^\star)>0\).
			Denoting \(C\coloneqq\rho_{\rm max}(1+\rp\rho_{\rm max})\), we have
			\begin{align*}
				\gamma_{k-1} P_{k-1}
			\leq{} &
				\|x^{k-1}-x^\star\|
				\left(
					\|x^{k-1}-x^{k-2}\|
					+
					\gam_{k-1}\|\nabla f(x^{k-1})-\nabla f(x^{k-2})\|
				\right)
			\\
			={} &
				\|x^{k-1}-x^\star\|
				\left(
					\|x^{k-1}-x^{k-2}\|
					+
					\gam_{k-1}L_{k-1}\|x^{k-1}-x^{k-2}\|^{\q}
				\right)
			\end{align*}
			for every \(k\). Then, by \eqref{eq:gamkPk:U}
			\begin{align*}
				0
			<
				U
			={} &
				\lim_{K\ni k\to\infty}
				\gamk(1+\rp\rhok)P_{k-1}
			={}
				\liminf_{K\ni k\to\infty}
				\rhok(1+\rp\rhok)\gamma_{k-1}P_{k-1}
			\\
			\leq{} &
				\rho_{\rm max}(1+\rp\rho_{\rm max})
				\liminf_{K\ni k\to\infty}
				\|x^{k-1}-x^\star\|
				\left(
					\|x^{k-1}-x^{k-2}\|
					+
					\gam_{k-1}L_{k-1}\|x^{k-1}-x^{k-2}\|^{\q}
				\right)
			\end{align*}
			which yields the first claim owing to \eqref{eq:shifted:res} and \(\q>0\).

			This along with the update rule \eqref{eq:gamk*} implies
			\begin{align*}
				\rhok \leq
				\frac{
					1
				}{
					2\left[\gam_{k-1}^2L_{k-1}^2-(2-\rp)\gam_{k-1}\ell_{k-1} + 1-\rp\right]
				} \to 0,
			\end{align*}
			as claimed.
		\end{claims*}
		Having shown the above claims, the proof is concluded as in \cite[Thm. 2.4(ii)]{latafat2023convergence} by constructing a specific unbounded stepsize sequence and using claims \ref{claim*:PG:gamk*} and \ref{claim*:gamkLk:infty} to obtain the sought contradiction.
			\end{proof}
			\end{WhereToPutThis}

			While a global lower bound for the stepsizes \(\seq{\gamk}\) is not available, it is at the moment unclear whether one for the (entire) scaled sequence \(\seq{\lamk}\) exists.
			Nevertheless, with \(P_k\) as in \eqref{eq:Pk}, a lower bound for an alternative scaled sequence \(\seq{\gamk*P_k^{-\nicefrac{(1-\q)}{\q}}}\) does exist, thanks to which the following convergence rate can be achieved.

			\begin{theorem}[sublinear rate]\label{thm:sublinear}%
				Suppose that \cref{ass:basic} holds.
				Then, the following sublinear rate holds for the iterates generated by \adaPG:
				\[\mathtight[0.9]
					\min_{0\leq i\leq K}P_i
				\leq
					\max\set{
						\frac{\U_0(x^\star)}{\gam_0(K+1)},
						\frac{C(\rp,\q)\U_0(x^\star)^{\frac{1+\q}{2}}\L_{\Omega}}{(K+1)^{\q}}
					}
				\]
				where
				\(
					C(\rp,\q)
				=
					\sqrt{2}\left(\sqrt{\rp}\right)^{\q}\bigl(\sqrt{2}\rho_{{\rm max}}+1\bigr)^{1-\q}
				\)
				and \(\L_\Omega\) is a \(\q\)-H\"older modulus for \(\nabla f\) on a compact convex set \(\Omega\) that contains all the iterates \(x^k\).
			\end{theorem}
			\begin{WhereToPutThis}
			\begin{proof}

		The existence of \(\Omega\) as in the statement follows from boundedness of \(\seq{x^k}\), see \cref{thm:bounded}.
		We proceed by intermediate claims.
		\def\currentlabel{thm:sublinear}%
		\begin{claims}
		\item \label{thm:sublinear:Pk}%
			{\em If \(\lamk\leq\frac{1}{\L_\Omega}\), then \(P_k \leq P_{k-1}\).}

			Let \(\nabla*\varphi(x^k)\coloneqq\nabla f(x^k)+\nabla*g(x^k)\), where \(\nabla*g\) is as in \eqref{eq:Dg}.
			Then, \(\nabla*\varphi(x^k)\in\partial\varphi(x^k)\) and thus
			\begin{align*}
				\varphi(x^{k-1})
			\geq{} &
				\varphi(x^k)
				+
				\innprod{\nabla*\varphi(x^k)}{x^{k-1}-x^k}
			\\
			={} &
				\varphi(x^k)
				+
				\tfrac{1}{\gamk}
				\innprod{\Hk(x^{k-1})-\Hk(x^k)}{x^{k-1}-x^k}
			\\
			={} &
				\varphi(x^k)
				+
				\tfrac{1}{\gamk}\|x^k-x^{k-1}\|^2
				-
				\Lk\|x^{k-1}-x^k\|^{1+\q}
			\\
			={} &
				\varphi(x^k)
				+
				\bigl(
					\tfrac{1}{\lamk}
					-
					\Lk
				\bigr)
				\|x^{k-1}-x^k\|^{1+\q},
			\end{align*}
			establishing the claim.

			We next aim at establishing a lower bound on the stepsize sequence in terms of \(P_k^{\frac{1-\q}{\q}}\).
			To simplify the exposition, we now fix \(x^\star\in\argmin\varphi\) and denote
			\begin{equation}\label{eq:Cnu}
				\tilde C_{\q}^\rp(\nu)
			\coloneqq
				\sqrt{2\U_1(x^\star)}
				\bigl(\tfrac{1}{\nu}+\L_{\Omega}\bigr).
			\end{equation}

		\item \label{thm:sublinear:Pk<Dk}%
			{\em
				For every \(k\in\N\) it holds that
				\(
					P_k
				\leq
					\tilde C_{\q}^\rp(\lamk)
					\|x^k-x^{k-1}\|^{\q}
				\).
			}

			We begin by observing that
			\[
				\nabla*\varphi(x^k)
			\coloneqq
				\tfrac{1}{\gamk}
				\bigl(\Hk(x^{k-1})-\Hk(x^k)\bigr)
			\in
				\partial\varphi(x^k)
			\]
			owing to \eqref{eq:Dg}.
			Combined with \eqref{eq:hkq:bound} and \eqref{eq:lamk} it follows that
			\[
				\|\nabla* \varphi(x^k)\|
			\leq
				\bigl(\tfrac{1}{\lamk}+\Lk\bigr)
				\|x^k-x^{k-1}\|^{\q}
			\leq
				\bigl(\tfrac{1}{\lamk}+\L_\Omega\bigr)
				\|x^k-x^{k-1}\|^{\q}.
			\]
			Moreover, by convexity,
			\[
				P_k
			=
				\varphi(x^k)
				-
				\min\varphi
			\leq
				\innprod{\nabla*\varphi(x^k)}{x^k-x^\star}
			\leq
				\|\nabla*\varphi(x^k)\|
				\|x^k-x^\star\|
			\leq
				\overbracket*[0.5pt]{
					\sqrt{2\U_1(x^\star)}
					\bigl(\tfrac{1}{\lamk}+\L_\Omega\bigr)
				}^{\tilde C_{\q}^\rp(\lamk)}
				\|x^k-x^{k-1}\|^{\q}
			\]
			as claimed, where the last inequality uses the fact that \(\tfrac12\|x^k-x^\star\|^2\leq\U_k(x^\star)\leq\U_1(x^\star)\).

			We next analyze two possible cases for any iteration index \(k\geq 0\).

		\item \label{claim:Pk:bound:2}%
			{\em
				For any \(k\in K_2\),~
				\(
					\gamk*
				\geq
					\frac{1}{\sqrt{2}\L_\Omega}
					\left(
						\frac{P_k}{\tilde C_{\q}^\rp(\lam_{\rm min})}
					\right)^{\frac{1-\q}{\q}}
				\).
			}%

			Since \(\Lk\leq\L_{\Omega}\), as shown in \cref{thm:lammin}
			\(
				\lamk
			\geq
				\lam_{\rm min}
			=
				\frac{1}{\sqrt2\L_{\Omega}\rho_{\rm max}}
			\)
			holds for every \(k\in K_2\).
			Moreover, by definition of \(K_2\),
			\begin{align}
			\nonumber
				\gamk*
			={} &
				\frac{
					\gamk
				}{
					\sqrt{2\left[\lamk^2\Lk^2-(2-\rp)\lamk\lk + 1- \rp\right]_+}
				}
			\\
			={} &
				\frac{
					\lamk\|x^k-x^{k-1}\|^{1-\q}
				}{
					\sqrt{2\left[\lamk^2\Lk^2-(2-\rp)\lamk\lk + 1- \rp\right]_+}
				}
			\geq
				\frac{\|x^k-x^{k-1}\|^{1-\q}}{\sqrt{2}\Lk}
			\geq
				\frac{\|x^k-x^{k-1}\|^{1-\q}}{\sqrt{2}\L_\Omega}
		\ifarxiv\else
			\quad
				\forall k\in K_2.
		\fi
			\label{eq:lamk*:Dk}
			\end{align}
		\ifarxiv
			holds for all \(k\in K_2\).
		\fi
			By using the lower bound
		\ifarxiv\[\else\(\fi
				\|x^k-x^{k-1}\|^{1-\q}
			\geq
				\Bigl(\frac{P_k}{\tilde C_{\q}^\rp(\lamk)}\Bigr)^{\frac{1-\q}{\q}}
			\geq
				\Bigl(\frac{P_k}{\tilde C_{\q}^\rp(\lam_{\rm min})}\Bigr)^{\frac{1-\q}{\q}}
		\ifarxiv\]\else\)\fi
			in \cref{thm:sublinear:Pk<Dk} raised to the power \(\frac{1-\q}{\q}\) the claim follows.

		\item \label{claim:Pk:bound:1}%
			{\em
				For any \(k\in K_1\),~
				\(
					\gamk*
				\geq
					\begin{ifcases}
						\frac{1}{\sqrt{2\rp}\L_\Omega}
						\bigl(
							\frac{P_k}{\tilde C_{\q}^\rp(\lam_{\rm min})}
						\bigr)^{\frac{1-\q}{\q}}  & \text{(\(K_2\neq\emptyset\) and) } k \geq\min K_2,
					\\[3pt]
						\bigl(
							1 + \tfrac{1}{\rp}
						\bigr)^{\frac{k}{2}} \gam_0
					\otherwise.
					\end{ifcases}
				\)%
			}

			Let
			\[
				K_{2,<k}
			\coloneqq
				K_2\cap\set{0,1,\dots,k-1}
			\]
			denote the (possibly empty) set of all iteration indices up to \(k-1\) such that the first term in \eqref{eq:gamk*} is strictly larger than the second one.

			If \(K_{2,<k} = \emptyset\), then
			\(
				\rho_{t+1} = \sqrt{\nicefrac{1}{\rp}+\rho_t}
			\)
			holds for all \(t\leq k\), which inductively gives
			\(
				\rho_t^2
			\geq
				1 + \tfrac{1}{\rp}
			\)
			for all \(t=1,\dots,K\) (since \(\rho_0 \geq 1\)).
			We then have
			\begin{equation}\label{eq:PG:gam2k}
				\gamk*^2
			=
				\gam_0^2\textstyle\prod_{t=1}^k\rho_{t+1}^2
			\geq
				(1+ \tfrac{1}{\rp})^k\gam_0^2.
			\end{equation}

			Suppose instead that \(K_{2,<k}\neq\emptyset\), and let \(n_{2,k}\) denote its largest element:
			\[
				n_{2,k}
			\coloneqq
				\max K_{2,<k}
			=
				\max\set{i<k}[
					\gam_{i+1}
				<
					\gam_i\sqrt{\tfrac{1}{\rp}+\rho_i}
				].
			\]

			Observe that the update rule \(\rho_{i+1}=\sqrt{\frac{1}{\rp}+\rho_i}\) implies that \(\rho_{i+2}\geq1\) holds whenever \(i,i+1\in K_1\).
			In fact,
			\(
				\rho_{i+1}=\sqrt{\tfrac{1}{\rp}+\rho_i}\geq\tfrac{1}{\sqrt{\rp}}
			\)
			holds for every \(i\in K_1\), in turn implying that
			\[
			\textstyle
				i,i+1\in K_1
			\quad\Rightarrow\quad
				\rho_{i+2}
			\geq
				\sqrt{\frac{1}{\rp}+\sqrt{\frac{1}{\rp}}}
			\geq
				\sqrt{\frac{1}{2}+\sqrt{\frac{1}{2}}}
			>
				1.
			\]

			In particular,
			\begin{equation}\label{eq:prodK1}
			\textstyle
				i,i+1,\dots,j\in K_1
			\quad\Rightarrow\quad
				\prod_{t=i+1}^{j+1}\rho_t
				\geq
				\tfrac{1}{\sqrt{\rp}}
			\end{equation}
			(this being also trivially true for an empty product, since \(\rp\geq1\)).

			We consider two possible subcases:
			\begin{itemize}[label=\(\diamond\),leftmargin=*]
			\item
				First, suppose that the index \(j\coloneqq\max \set{n_{2,k}\leq i \leq k}[\lam_{i}>\tfrac{1}{\L_\Omega}]\) exists.
				Schematically,
				\begin{equation}\label{eq:jscheme}
					\overbracket[0.5pt]{\,
						n_{2,k}
						\vphantom{jk}
					\,}^{\in K_2}
					,
					\overbracket[0.5pt]{\,
						\dots,j
						,
						\underbracket[0.5pt]{\,
							\dots,k
						\,}_{\mathclap{\lam_i\leq\frac{1}{\L_\Omega}}}
					\,}^{\in K_1}
				\quad\text{and}\quad
					\lam_j>\tfrac{1}{\L_\Omega}.
				\end{equation}
				By definition of \(n_{2,k}\), all indices between \(j\) and \(k\) are in \(K_1\), and thus
				\begin{align*}
					\gamk*
				=
					\gam_j
					\prod*_{i=j+1}^{k+1}\rho_i
				\overrel[\geq]{\eqref{eq:prodK1}}{} &
					\tfrac{1}{\sqrt{\rp}}
					\gam_j
				=
					\tfrac{1}{\sqrt{\rp}}
					\lam_j
					\|x^j-x^{j-1}\|^{1-\q}
				\\
				\overrel[>]{\eqref{eq:jscheme}}{} &
					\tfrac{1}{\sqrt{\rp}}
					\tfrac{1}{\L_{\Omega}}
					\|x^j-x^{j-1}\|^{1-\q}
				\overrel[\geq]{\cref{thm:sublinear:Pk<Dk}}[3pt]
					\tfrac{1}{\sqrt{\rp}}
					\tfrac{1}{\L_{\Omega}}
					\Bigl(
						\tfrac{P_j}{\tilde C_{\q}^\rp(\lam_j)}
					\Bigr)^{\frac{1-\q}{\q}}
				\overrel[>]{\eqref{eq:jscheme}}
					\tfrac{1}{\sqrt{\rp}}
					\tfrac{1}{\L_{\Omega}}
					\Bigl(
						\tfrac{P_j}{\tilde C_{\q}^\rp(\nicefrac{1}{\L_\Omega})}
					\Bigr)^{\frac{1-\q}{\q}}.
				\end{align*}
				Since \(\lam_i \leq \tfrac{1}{\L_{\Omega}}\) holds for all \(i=j+1,\dots,k\), it follows from \cref{thm:sublinear:Pk} that \(P_k\leq P_j\), and thus
				\begin{equation}\label{eq:Pk:bound:1a}
					\gamk*
				\geq
					\tfrac{1}{\sqrt{\rp}}
					\tfrac{1}{\L_{\Omega}}
					\Bigl(
						\tfrac{P_k}{\tilde C_{\q}^\rp(\nicefrac{1}{\L_\Omega})}
					\Bigr)^{\frac{1-\q}{\q}}.
				\end{equation}

			\item
				Alternatively, it holds that \(\lam_{j} \leq \tfrac{1}{\L_{\Omega}}\) for all \(j=n_{2,k},\dots,k\), and in particular by virtue of \cref{thm:sublinear:Pk<Dk} we have that \(P_k\leq P_{n_{2,k}}\).
				Arguing as before,
				\begin{equation}\label{eq:Pk:bound:1b}
					\gamk*
				=
					\gam_{n_{2,k}+1}
					\prod*_{i=n_{2,k}+1}^{k+1}\rho_i
				\overrel[\geq]{\eqref{eq:prodK1}}
					\tfrac{1}{\sqrt{\rp}}
					\gam_{n_{2,k}+1}
				\overrel[\geq]{\cref{claim:Pk:bound:2}}[3pt]
					\tfrac{1}{\sqrt{2\rp}\L_\Omega}
					\left(
						\tfrac{P_{n_{2,k}}}{\tilde C_{\q}^\rp(\lam_{\rm min})}
					\right)^{\frac{1-\q}{\q}}
				\geq
					\tfrac{1}{\sqrt{2\rp}\L_\Omega}
					\left(
						\tfrac{P_k}{\tilde C_{\q}^\rp(\lam_{\rm min})}
					\right)^{\frac{1-\q}{\q}}.
				\end{equation}
			\end{itemize}
			Combining \eqref{eq:Pk:bound:1a} and \eqref{eq:Pk:bound:1b}
			\[
				\gamk*
			\geq
				\min\set{
					\frac{1}{\tilde C_{\q}^\rp(\nicefrac{1}{\L_\Omega})^{\frac{1-\q}{\q}}}
				,~
					\frac{1}{\sqrt{2}\tilde C_{\q}^\rp(\lam_{\rm min})^{\frac{1-\q}{\q}}}
				}
				\frac{P_k^{\frac{1-\q}{\q}}}{\sqrt{\rp}\L_{\Omega}}
			=
				\frac{1}{\sqrt{2\rp}\L_{\Omega}\tilde C_{\q}^\rp(\lam_{\rm min})^{\frac{1-\q}{\q}}}
				P_k^{\frac{1-\q}{\q}},
			\]
			where the identity uses the fact that the minimum is attained at the first element, having \(\tilde C_{\q}^\rp(\nu)\) decreasing in \(\nu>0\) and \(\frac{1}{\L_\Omega}\geq\lam_{\rm min}=\frac{1}{\sqrt{2}\rho_{\rm max}\L_\Omega}\) (since \(\rho_{\rm max}\geq1\)).
		\end{claims}

		Finally, combining \cref{claim:Pk:bound:1,claim:Pk:bound:2} and noting that
		\[
		\frac{1}{\sqrt{2\rp}\L_\Omega}
		\left(
			\frac{P_k}{\tilde C_{\q}^\rp(\lam_{\rm min})}
			\right)^{\frac{1-\q}{\q}}
		=
			\frac{1}{\sqrt{2\rp}\L_\Omega^{\nicefrac{1}{\q}}}
			\left(
				\frac{1}{
					\sqrt{2\U_1(x^\star)}
					\bigl(\sqrt2 \rho_{\rm max}+1\bigr)
				}
			\right)^{\frac{1-\q}{\q}}
			P_k^{\frac{1-\q}{\q}},
		\]
		we conclude that
		\[
			\gamk*
		\geq
			\begin{ifcases}
				\displaystyle
				\frac{1}{
					\U_1(x^\star)^{\frac{1-\q}{2\q}}
					C(\rp,\q)^{\frac{1}{\q}}
				}
				P_k^{\frac{1-\q}{\q}}  & \text{(\(K_2\neq\emptyset\) and) } k \geq\min K_2,
			\\[15pt]
				\bigl(1 + \tfrac{1}{\rp}\bigr)^{\frac{k}{2}} \gam_0  \otherwise
			\end{ifcases}
		\]
		holds for any \(k\in\N\),
		where
		\[
			C(\rp,\q)
		=
			\sqrt{2}\L_\Omega
			\sqrt{\rp}^{\q}
			\bigl(
				1+\sqrt{2}\rho_{\rm max}
			\bigr)^{1-\q}
		\]
		is as in the statement.
		Denoting \(k_0=\min K_2-1\) if \(K_2\neq\emptyset\) and \(0\) otherwise, the sum of stepsizes can be lower bounded by
		\begin{align*}
		\textstyle
			\sum_{k=1}^{K+1} \gamk
		=
			\sum_{k=1}^{k_0} \gamk  + \sum_{k=k_0}^{K+1} \gamk
		\geq{} &
			\gam_0\sum_{k=1}^{k_0} (1 + \tfrac{1}{\rp})^{\frac{k}{2}}
			+
			\frac{1}{
				\U_1(x^\star)^{\frac{1-\q}{2\q}}
				C(\rp,\q)^{\frac{1}{\q}}
			}
			\sum_{k=k_0}^{K+1} P_{k-1}^{\frac{1-\q}{\q}}
		\\
		\geq{} &
			\gam_0\sum_{k=1}^{k_0} (1 + \tfrac{1}{\rp})^{\frac{k}{2}}
			+
			\frac{K+1-k_0}{
				\U_1(x^\star)^{\frac{1-\q}{2\q}}
				C(\rp,\q)^{\frac{1}{\q}}
			}
			\bigl(\min_{k \leq K}P_k\bigr)^{\frac{1-\q}{\q}}
		\\
		\geq{} &
			\gam_0 k_0
			+
			\frac{K+1-k_0}{
				\U_1(x^\star)^{\frac{1-\q}{2\q}}
				C(\rp,\q)^{\frac{1}{\q}}
			}
			\bigl(\min_{k \leq K}P_k\bigr)^{\frac{1-\q}{\q}}
		\\
		\geq{} &
			\min\set{
				\gam_0,
			~
				\frac{1}{
					\U_1(x^\star)^{\frac{1-\q}{2\q}}
					C(\rp,\q)^{\frac{1}{\q}}
				}
				\bigl(\min_{k \leq K}P_k\bigr)^{\frac{1-\q}{\q}}
			}
			(K+1).
		\end{align*}
		Therefore, in light of \cref{thm:sumgamk}, for every \(K\geq1\) we have
		\[
			\U_1(x^\star)
		\geq
			\min_{k\leq K}P_k
			\,
			\sum_{k=1}^{K+1}\gamk
		\geq
			\min\set{
				\gam_0\min_{k \leq K}P_k,
				\frac{1}{
					\U_1(x^\star)^{\frac{1-\q}{2\q}}
					C(\rp,\q)^{\frac{1}{\q}}
				}
				\bigl(\min_{k \leq K}P_k\bigr)^{\frac{1}{\q}}
			}
			(K+1).
		\]
		Equivalently, for every \(K\geq1\) it holds that
		\[
			\text{either}\quad
			\min_{k \leq K}P_k
		\leq
			\frac{
				\U_1(x^\star)
			}{
				\gam_0(K+1)
			}
		\quad\text{or}\quad
			\min_{k \leq K}P_k
		\leq
			\frac{
				\U_1(x^\star)^{\frac{1+\q}{2}}
				C(\rp,\q)
			}{
				(K+1)^{\q}
			}.
		\]
		Further using the fact that \(\U_1(x^\star)\leq\U_0(x^\star)\) by \cref{lemma:main:inequality} results in the claimed bound.
			\end{proof}
			\end{WhereToPutThis}

			Some remarks are in order regarding the convergence results of the method.
			First, the obtained rate matches the one of the standard convex H\"older smooth setting of \cite{bredies2008forward} and the one in the nonconvex case \cite[Prop. 9]{bolte2023backtrack}, while it is worse than the one of the Universal Primal Gradient Method in \cite{nesterov2015universal}.
			Since our analysis relies upon a more involved Lyapunov function along with an adaptive stepsize rule, whether or not it can be tightened in order to obtain a better rate remains an interesting open question.

			In the locally Lipschitz setting \(\q=1\), the above rate matches the \(O(\nicefrac{1}{(K+1)})\) of \cite[Thm. 1.1]{latafat2023convergence} (with \(r = \nicefrac{\rp}2\)).
			Despite such a worst-case sublinear rate, the fast behavior of the algorithm in practice can be explained by utilization of large stepsizes and the bound in \cref{thm:sumgamk}.

			We also remark that under (local) strong convexity, up to modifying the stepsize update similarly to  \cite[\S2.3]{malitsky2020adaptive}, a contraction can be established in terms of \(\U_k(x^\star)\) in \cref{lemma:main:inequality}.
			We refer the reader to \cite{malitsky2020adaptive} for this approach and postpone a more detailed analysis to a future work.

	\section{Experiments}\label{sec:experiments}%

		\begin{figure*}[t!]
			\includegraphics[width=0.24\linewidth]{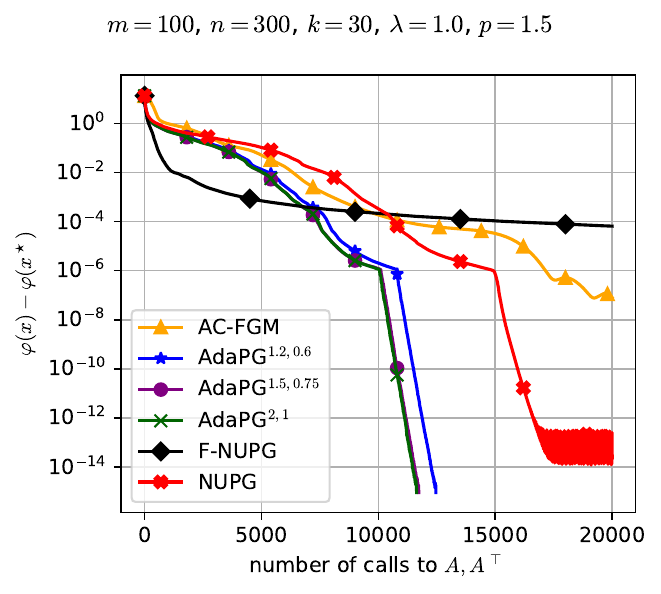}%
			\hfill
			\includegraphics[width=0.24\linewidth]{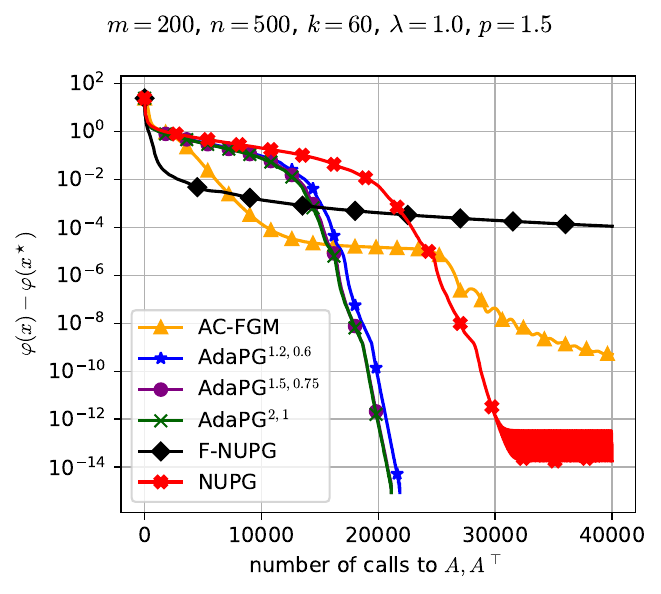}%
			\hfill
			\includegraphics[width=0.24\linewidth]{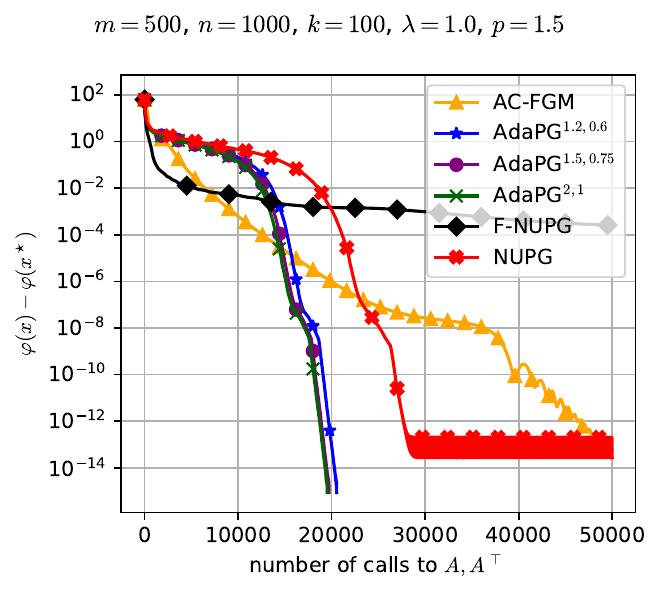}%
			\hfill
			\includegraphics[width=0.24\linewidth]{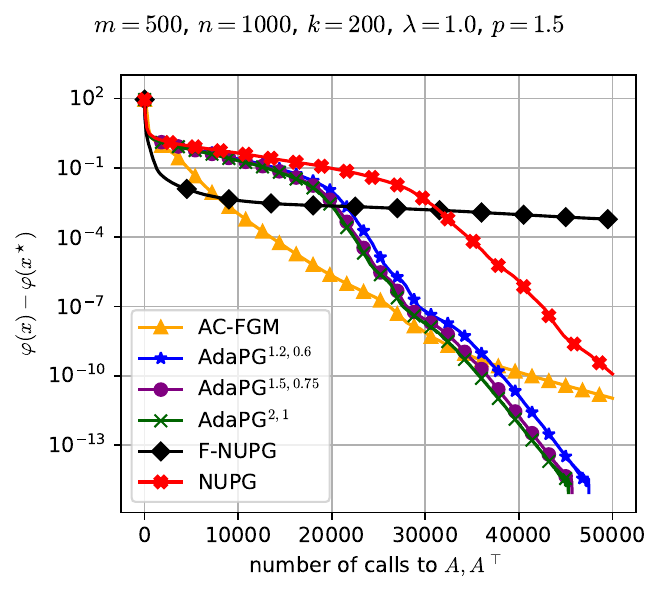}%

			\includegraphics[width=0.24\linewidth]{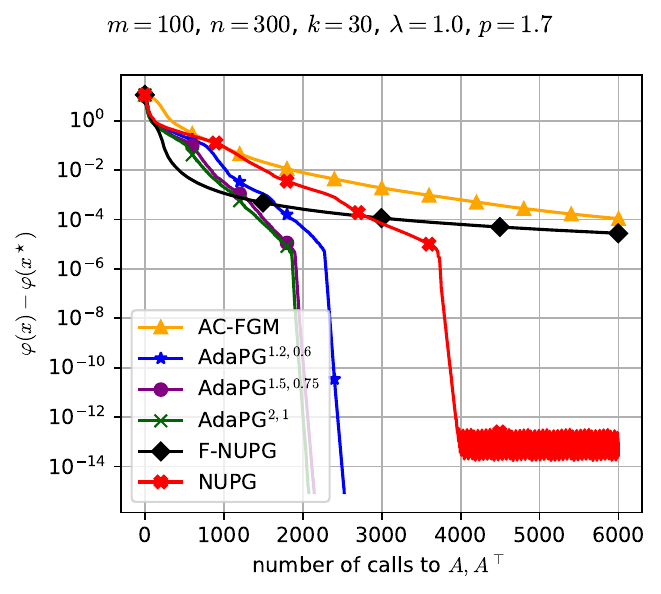}%
			\hfill
			\includegraphics[width=0.24\linewidth]{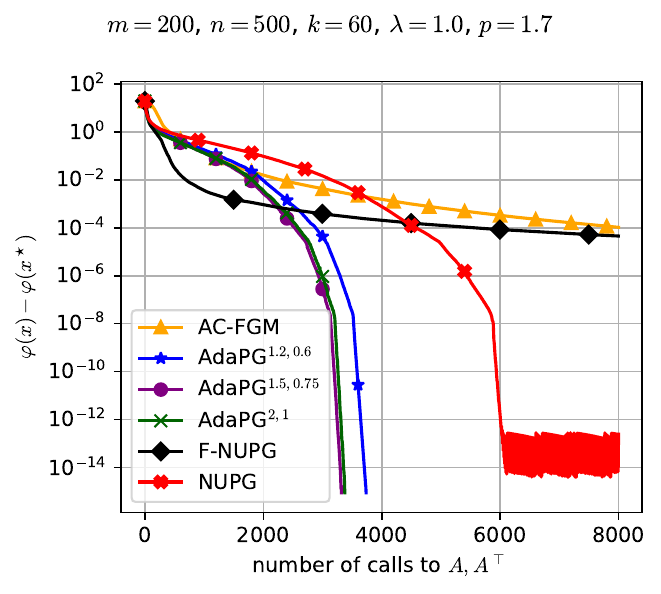}%
			\hfill
			\includegraphics[width=0.24\linewidth]{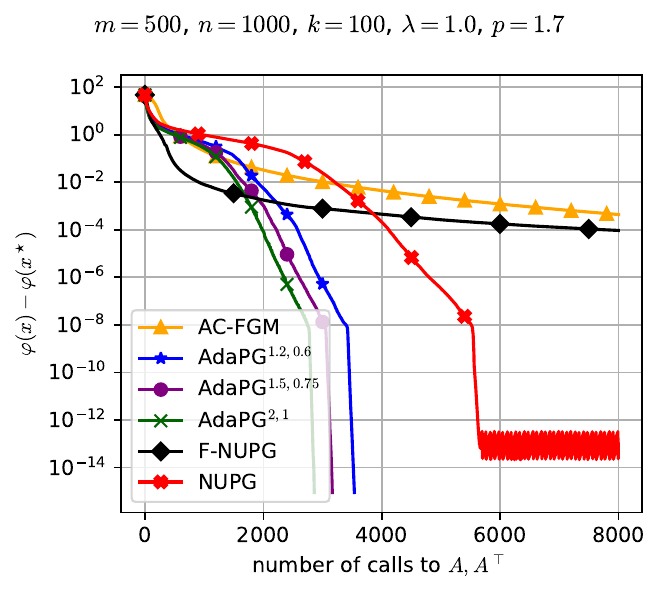}%
			\hfill
			\includegraphics[width=0.24\linewidth]{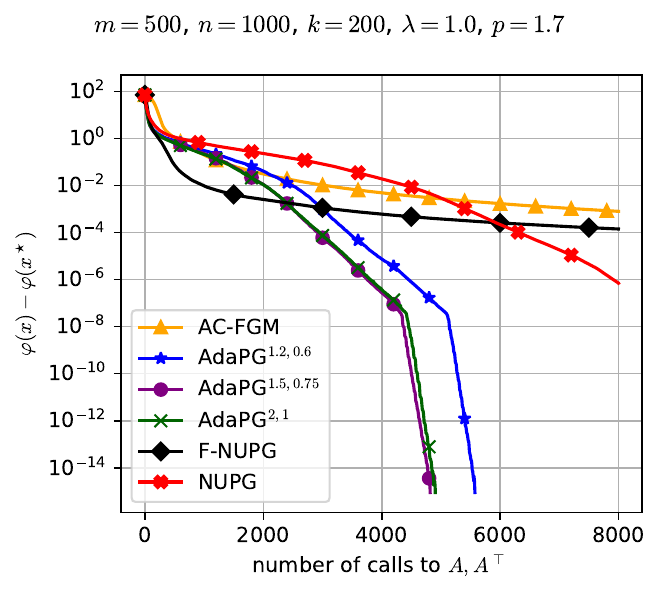}%

			\includegraphics[width=0.24\linewidth]{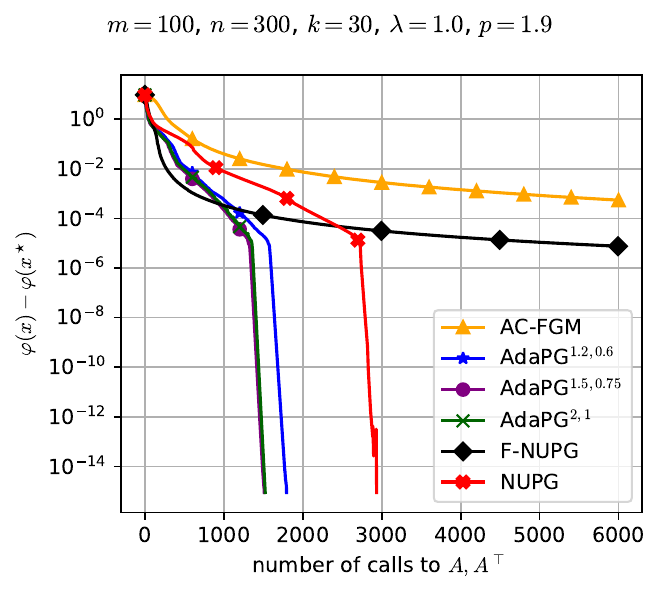}%
			\hfill
			\includegraphics[width=0.24\linewidth]{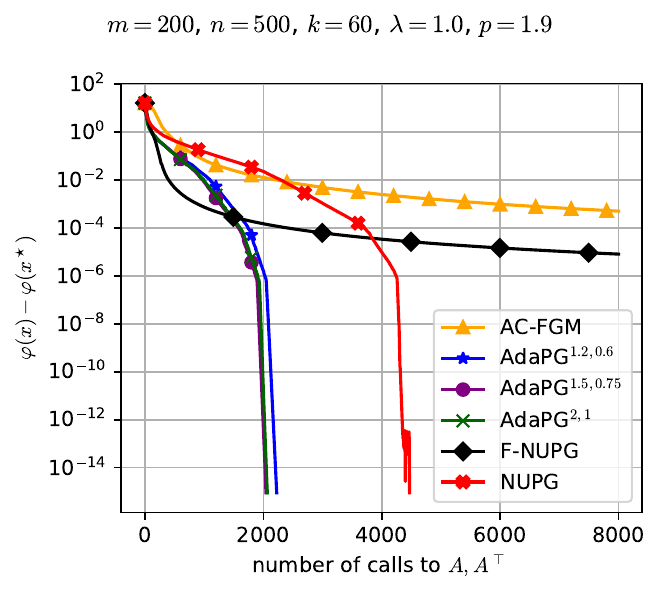}%
			\hfill
			\includegraphics[width=0.24\linewidth]{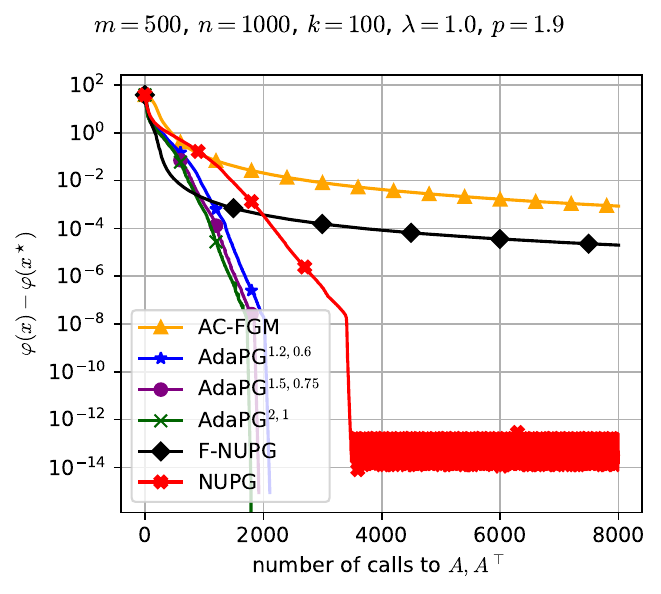}%
			\hfill
			\includegraphics[width=0.24\linewidth]{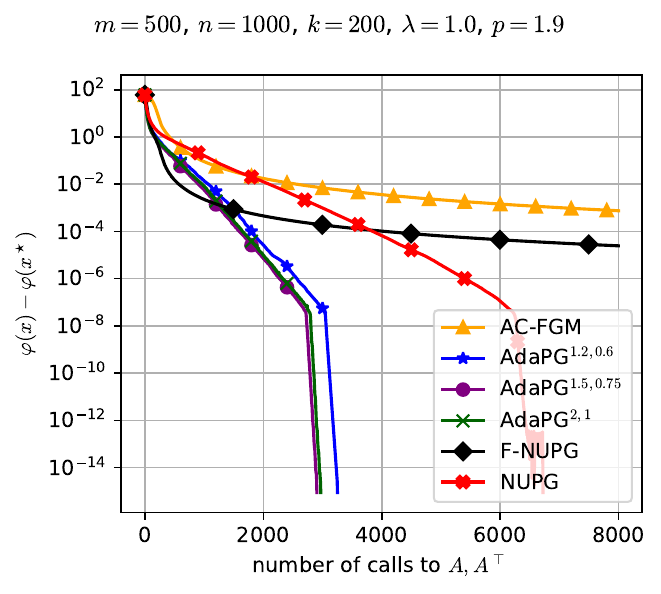}%
			\caption{%
				$p$-norm Lasso with varying powers $p$.
				It can be seen that \adaPG{} performs better than \NUPG{} in all cases in terms of calls to $A,A^\top$.
				In this experiment \adaPG{} also performs consistently better than the accelerated algorithms \ACFGM{} and \FNUPG{}.
			}%
			\label{fig:lasso}%
		\end{figure*}

		\begin{figure*}[ht]
				\includegraphics[width=0.24\linewidth]{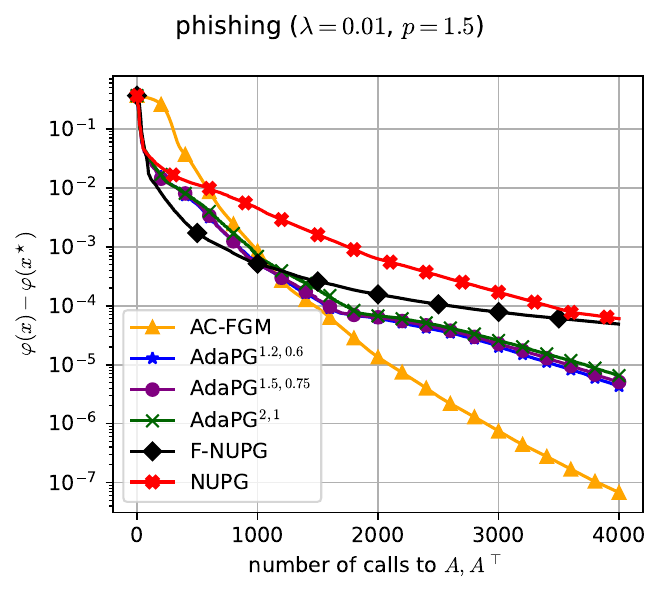}%
				\hfill
				\includegraphics[width=0.24\linewidth]{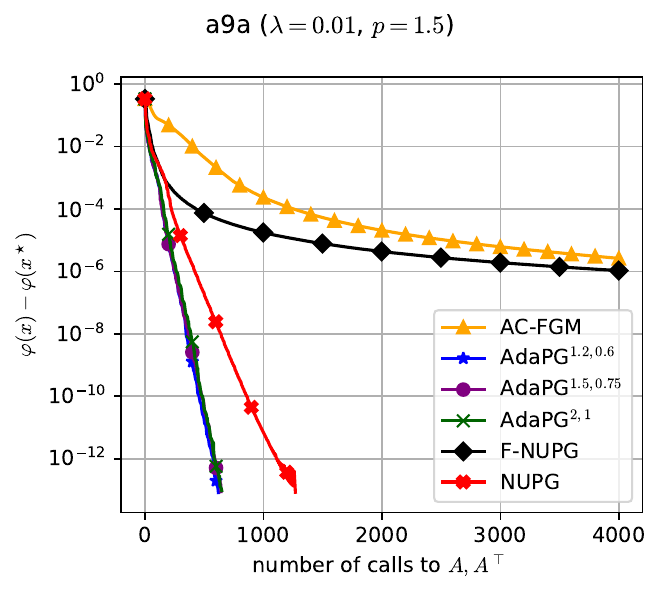}%
				\hfill
				\includegraphics[width=0.24\linewidth]{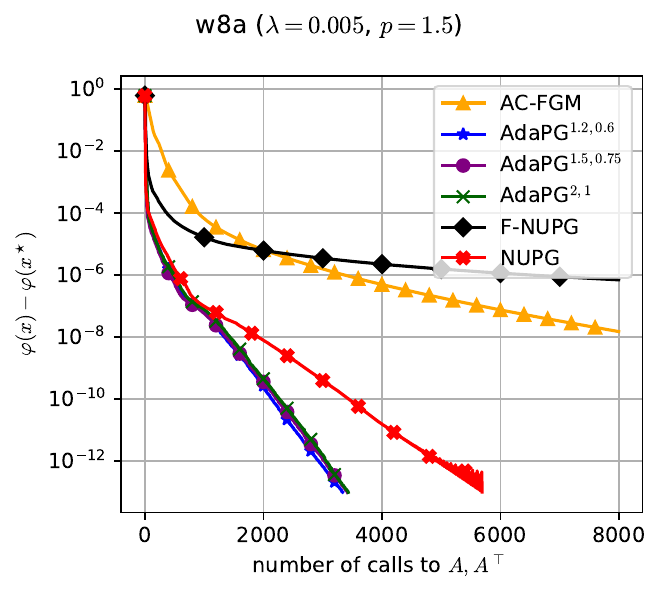}%
				\hfill
				\includegraphics[width=0.24\linewidth]{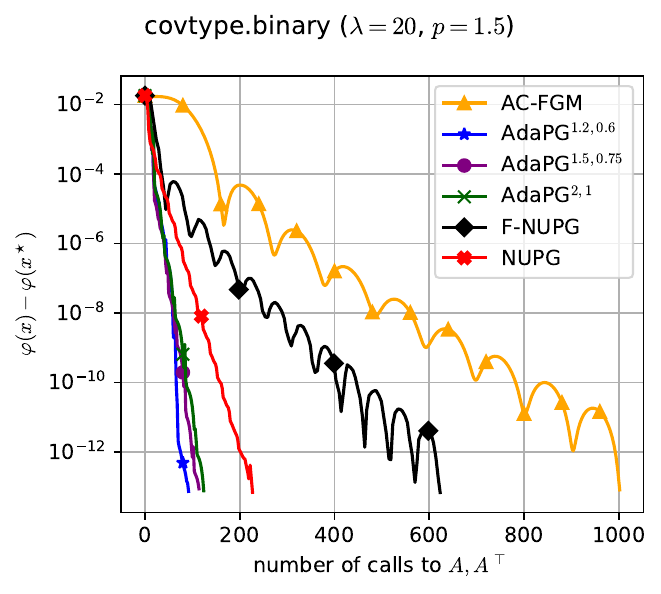}%

				\includegraphics[width=0.24\linewidth]{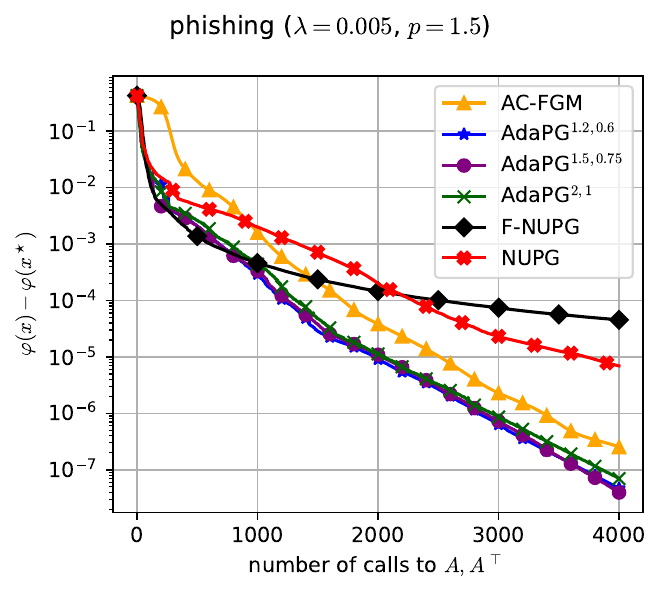}%
				\hfill
				\includegraphics[width=0.24\linewidth]{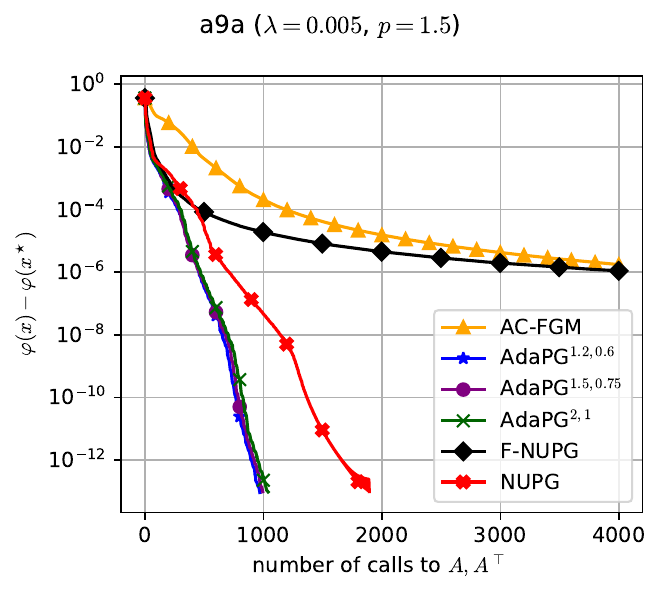}%
				\hfill
				\includegraphics[width=0.24\linewidth]{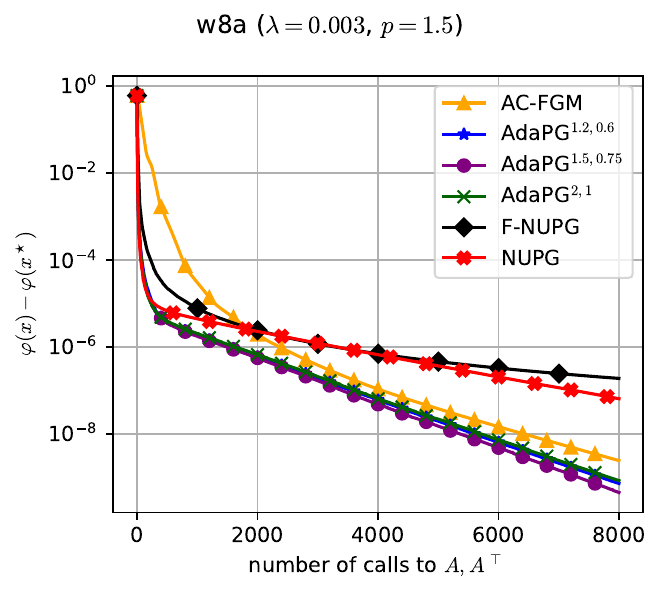}%
				\hfill
				\includegraphics[width=0.24\linewidth]{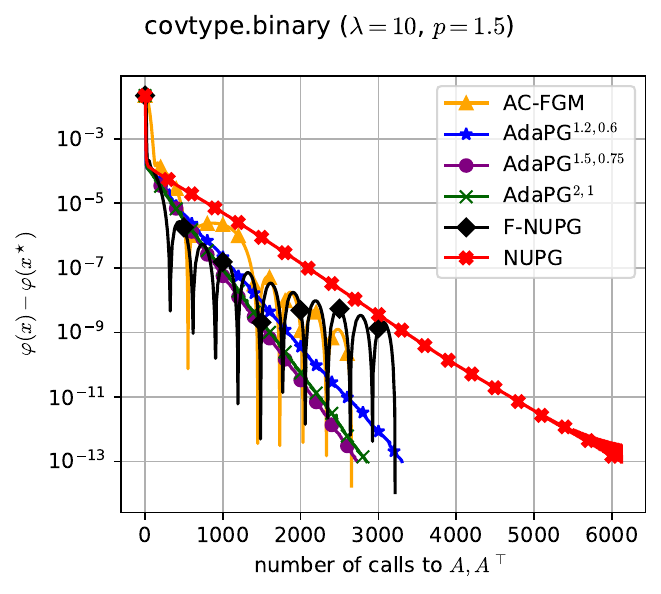}%

				\includegraphics[width=0.24\linewidth]{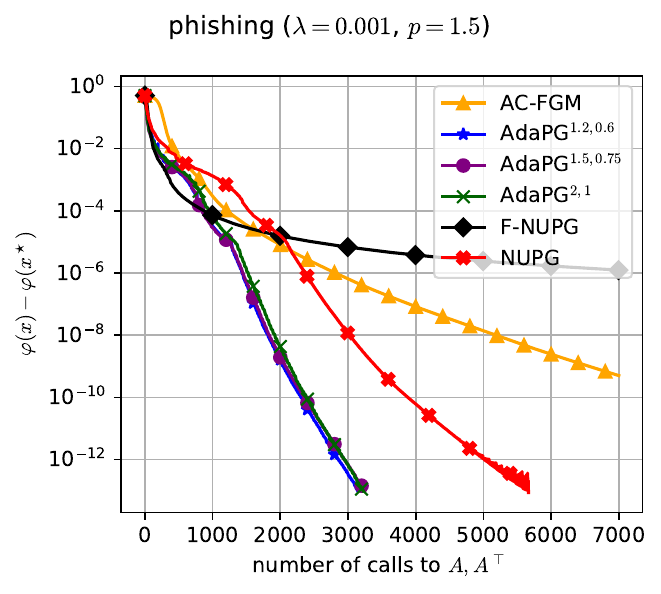}%
				\hfill
				\includegraphics[width=0.24\linewidth]{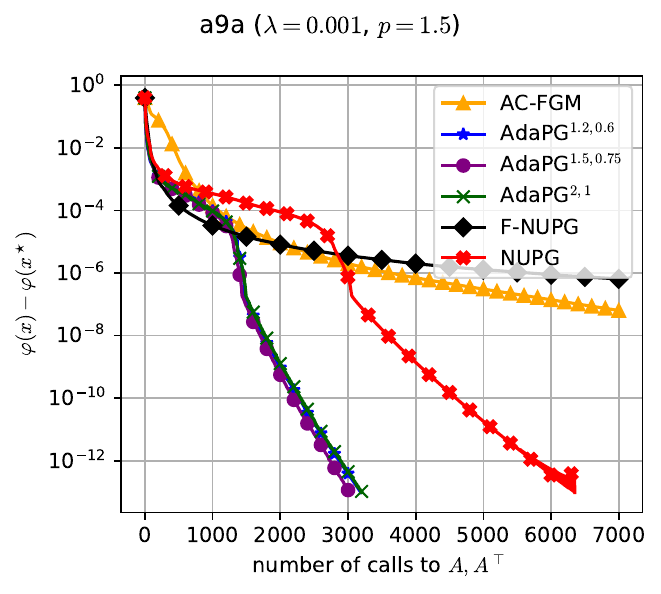}%
				\hfill
				\includegraphics[width=0.24\linewidth]{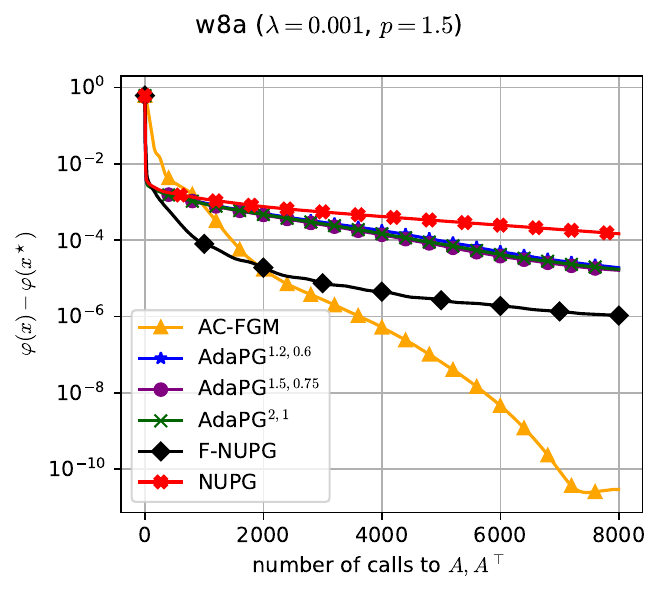}%
				\hfill
				\includegraphics[width=0.24\linewidth]{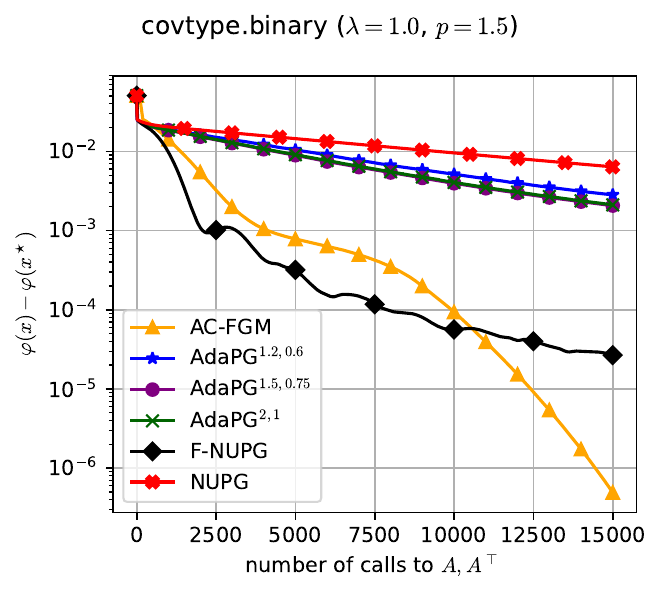}%

				\caption{%
					Classification with H\"older-smooth SVMs.
					It can be seen that \adaPG{} consistently outperforms \NUPG{} in terms of calls to $A,A^\top$.
				}%
				\label{fig:classification_svms}%
		\end{figure*}

		We demonstrate the performance of \adaPG{} on a range of simulations for three different choices of \(\rp\).
		We compare against the baseline and state-of-the-art methods: Nesterov's universal primal gradient method (\NUPG) \cite{nesterov2015universal}, its fast variant (\FNUPG) \cite{nesterov2015universal}, as well as the recently proposed auto-conditioned fast gradient method (\ACFGM) \cite{li2023simple}.

		While \adaPG{} only involves gradient evaluations, the other methods additionally require evaluations of the objective.
		In all applications considered in this section, the smooth part $f(x)$ takes the form ${\psi}(Ax)$ where $A$ is the design matrix containing the data.
		Consequently, matrix vector products $y=Ax$ and $A^\top \nabla{\psi}(y)${, each of complexity $\mathcal{O}(mn)$, constitute the most costly operations}.
		For the sake of a fair comparison, we store vectors that can be reused in subsequent evaluations, and plot the progress of the algorithm in terms of calls to $A$ and $A^\top$.
		As a result, denoting with \(\#_{A}\) the number of calls to \(A\) and \(A^\top\), and with \(\#_{\textsc{ls}} \geq 1\) the number of linesearch trials (including the successful ones), each iteration involves:
		\begin{itemize}[wide, topsep=0pt, itemsep=0pt]
		\item
			\NUPG:
			1 call to \(\nabla f\) and \(\#_{\textsc{LS}}\) to \(f\);
			by exploiting linearity, \(\#_{A} = 1 + \#_{\textsc{LS}}\)%

		\item
			\FNUPG:
			\(\#_{\textsc{LS}}\) calls to \(\nabla f\) and \(2\times \#_{\textsc{LS}}\) to \(f\);
			by exploiting linearity, \(\#_{A} = 3 + \#_{\textsc{LS}}\)%

		\item
			\AdaPG:
			1 call to \(\nabla f\);
			using linearity, \(\#_{A} = 2\)%

		\item
			\ACFGM:
			1 call to \(\nabla f\) and 1 to \(f\);
			by linearity, \(\#_{A} = 2\).%
		\end{itemize}

		\paragraph{Implementation details}
			In all experiments the solvers use $x=0$ as{}  the initial point, and{}  are executed with{}  the same initial stepsize computed as follows.
			First, we ran an offline proximal gradient update starting from the initial point and computing the stepsize as the inverse of \eqref{eq:L} evaluated between the initial and the obtained point.
			This procedure was repeated one additional time in case the obtained stepsize was substantially smaller than the original one.
			In the case of \ACFGM{}, in addition, we used the procedure of \cite{li2023simple} which can also be found in \cref{sec:algs}.
			The implementation for reproducing the experiments is publicly available.\footnote{\new{\url{https://github.com/EmanuelLaude/universal-adaptive-proximal-gradient}}}{}  		\subsection{\texorpdfstring{$p$}{p}-norm Lasso}
			In this experiment we consider a variant of Lasso in which the squared $2$-norm is replaced with a $p$-norm raised to some power $p \in (1, 2)$:
			\begin{equation} \label{eq:lasso}
				\minimize_{x \in \R^n} \tfrac{1}{p}\|Ax-b\|_p^p + \lambda \|x\|_1,
			\end{equation}
			for $A \in [-1,1]^{m \times n}$, $b \in \R^m$ with $n > m$ and $\lambda >0$.
			The first term in \eqref{eq:lasso} is differentiable with globally H\"older continuous gradient with order $\q=p-1$.
			The proximal mapping of the second term is the shrinkage operation which is computable in closed form.
			To assess the performance of the different algorithms we generate random instances of the problem using a $p$-norm version of the procedure provided in \cite{nesterov2013gradient}. In \cref{fig:lasso} we depict convergence plots for the different methods applied to random instances with varying dimensions of $A \in \R^{m \times n}$, powers $p \in (1, 2)$ and number of nonzero entries $k$ of the solution $x^\star$. It can be seen that \adaPG{} performs consistently better than Nesterov's universal primal gradient \cite{nesterov2015universal} method \NUPG{} with $\varepsilon=1e^{-12}$ in terms of evaluations of forward and backward passes (calls to $A,A^\top$). In this experiment \adaPG{} also performs consistently better than the accelerated algorithms \ACFGM{} and \FNUPG{}.

		\subsection{H\"older-smooth SVMs with \texorpdfstring{$\ell_1$}{l1}-regularization}
			In this subsection we assess the performance of the different algorithms on the task of classification using a $p$-norm version of the SVM model. For that purpose we consider a subset of the LibSVM binary classification benchmark that consists of a collection of datasets with varying number $m$ of examples $a_i$, feature dimensions $n$ and sparsity.
			Let $(a_j, b_j)_{j=1}^m$ with $b_j \in \{-1, 1\}$ and $a_j \in \R^n$ denote the collection of training examples. Then we consider the minimization problem
			\begin{equation}
				\minimize_{x \in \R^n} \frac{1}{m}\sum_{j=1}^m \tfrac{1}{p}\max\{0, 1 -b_j\langle a_j, x \rangle\}^p + \lambda \|x\|_1.
			\end{equation}
			for $p \in (1, 2)$. The loss function is globally H\"older smooth with order $\q=p-1$ while the proximal mapping of the second term can be solved in closed form. The results are depicted in \cref{fig:classification_svms}.

		\subsection{Logistic regression with \texorpdfstring{$p$}{p}-norm regularization}
			In this subsection we consider classification with the logistic loss and a smooth $p$-norm regularizer, for some \(p\in(1,2)\).
			The problem can be cast in convex composite minimization form as follows
			\begin{equation}
				\minimize_{x \in \R^n} \frac{1}{m}\sum_{j=1}^m \ln(1 -\exp(b_j\langle a_j, x \rangle)) + \lambda \tfrac{1}{p}\|x\|_p^p.
			\end{equation}
			Unlike the previous classification model this yields a smooth optimization problem.
			Hence we perform gradient steps with respect to $\varphi$ and choose the nonsmooth term $g\equiv 0$.
			The results are depicted in \cref{fig:classification_logistic}.
			\begin{figure}[t]
			 	\includegraphics[width=0.49\linewidth]{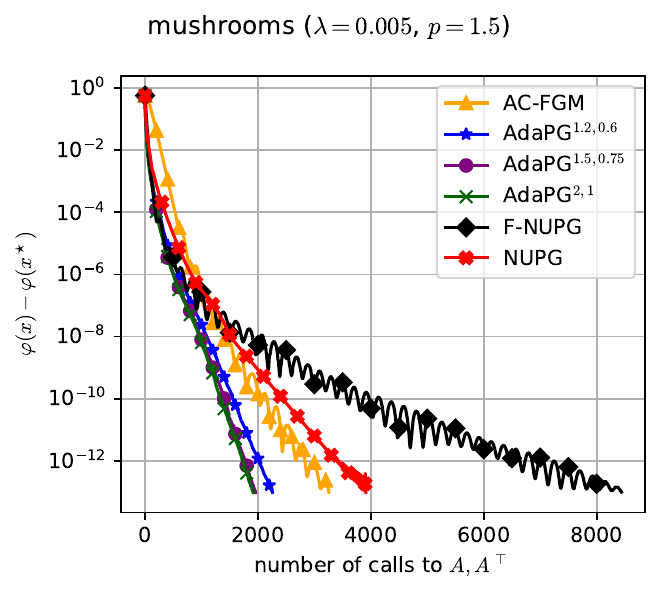}%
			 	\hfill
				\includegraphics[width=0.49\linewidth]{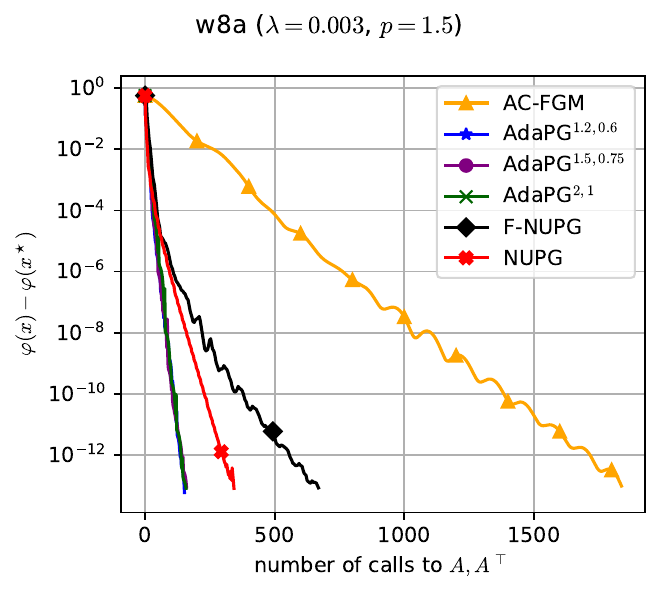}

				\includegraphics[width=0.49\linewidth]{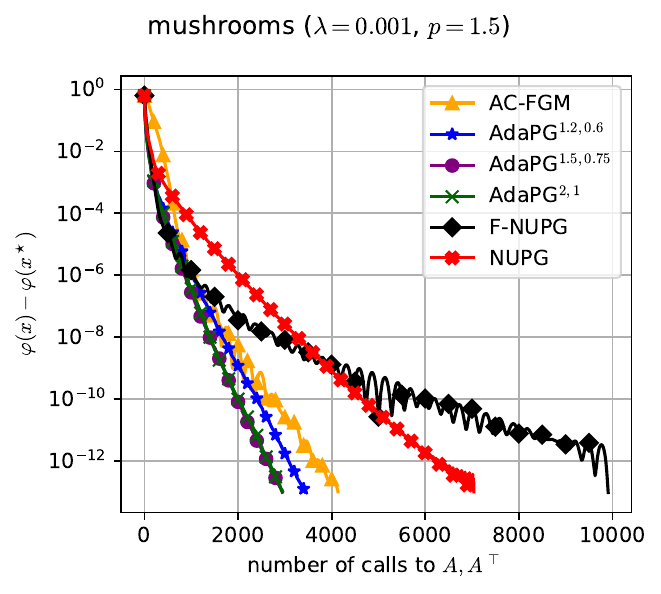}%
			 	\hfill
				\includegraphics[width=0.49\linewidth]{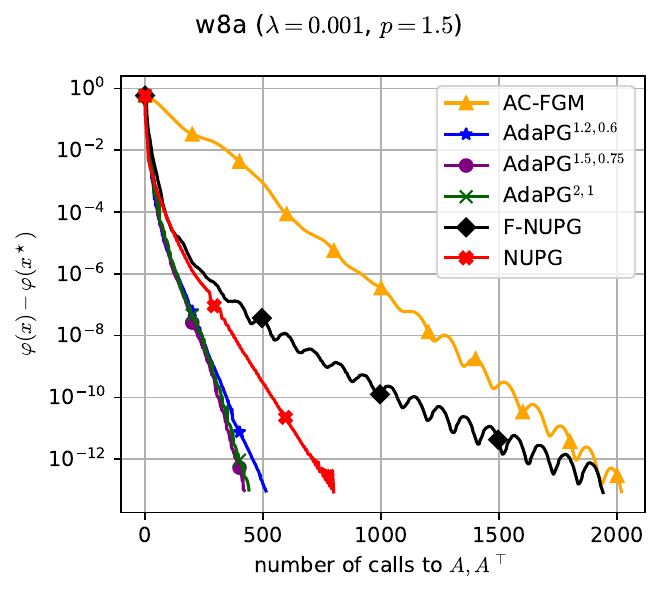}%
				\caption[]{%
					$p$-norm regularized logistic regression.
						In the present cases \adaPG{} performs better than the baseline algorithms.{}  					}%
				\label{fig:classification_logistic}%
			\end{figure}

		\subsection{Mixture \texorpdfstring{$p$}{p}-norm regression}
			In this final experiment we consider the following mixture model:
			\begin{equation} \label{eq:mixture}
				\minimize_{x \in \mathbb{B}_2(r)} \sum_{j=1}^J \tfrac{1}{p_j} \|A^j x - b^j\|_{p_j}^{p_j},
			\end{equation}
			where $\mathbb{B}_2(r)$ is the $2$-norm ball with radius $r$ and $p_j \in (1, 2]$.
			Since the $p_j$ are not identical the smooth part in \eqref{eq:mixture} is merely locally H\"older smooth.
			The nonsmooth part $g$ is the indicator function of the set $\mathbb{B}_2(r)$.
			In \cref{fig:mixture} we compare the performance for $J=6$ with $p = (1.8, 1_7, 1.6, 1_5, 1_5, 1_5)$ and different values of $n$ where the entries of $A^j \in [-1,1]^{m_j \times n}$ and $b^j \in [-1,1]^{m_j}$ are uniformly distributed between $[-1,1]$ with $m=(400, 300, 400, 100, 100, 300)$.

			\begin{figure}[ht!]
				\includegraphics[width=0.49\linewidth]{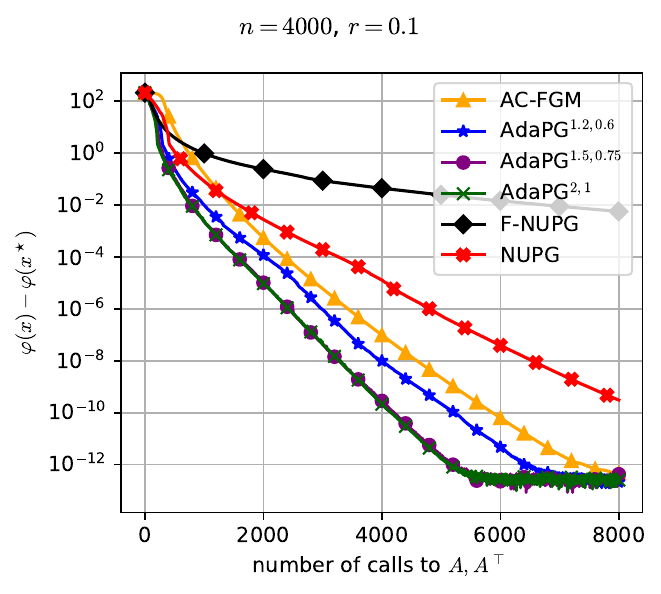}%
				\hfill
				\includegraphics[width=0.49\linewidth]{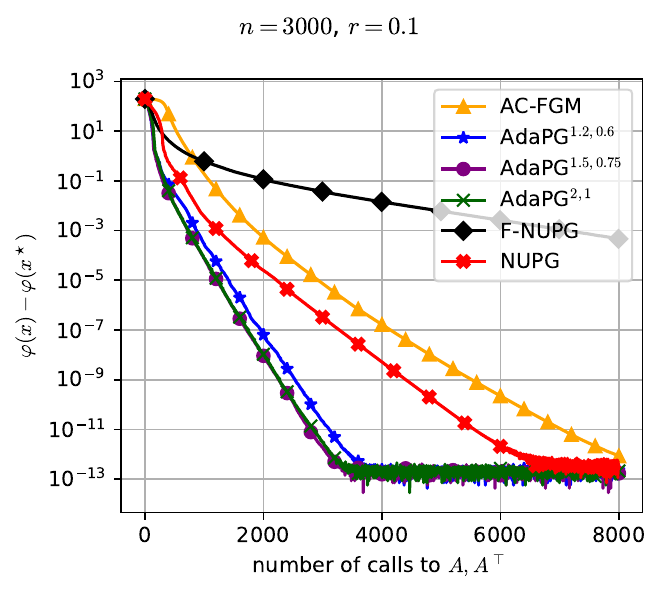}%

				\includegraphics[width=0.49\linewidth]{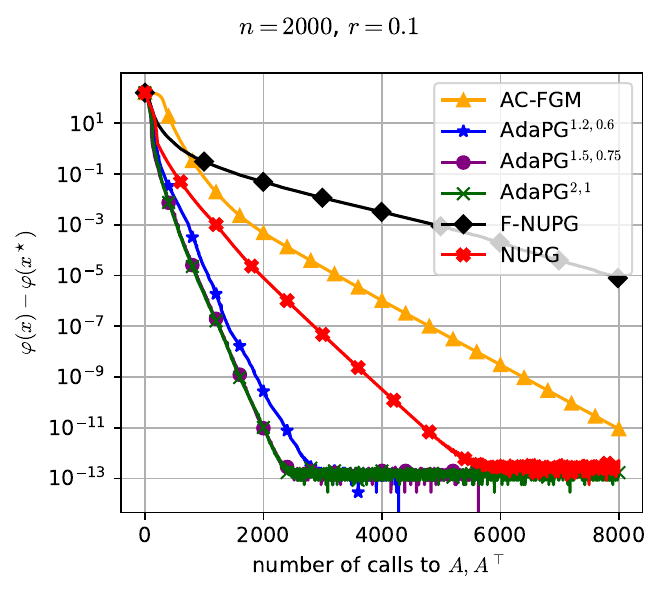}%
				\hfill
				\includegraphics[width=0.49\linewidth]{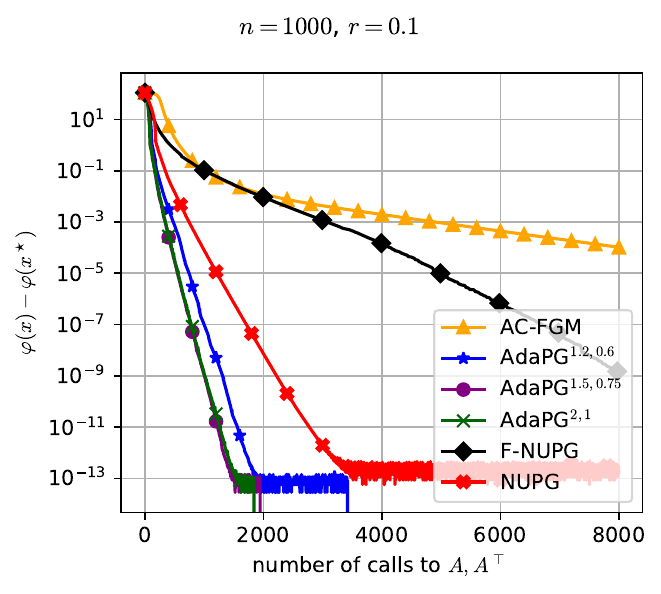}%
				\caption[]{%
					Mixture $p$-norm regression with ball constraint.
					It can be seen that \adaPG{} performs best on this comparison.
				}%
				\label{fig:mixture}%
			\end{figure}

	\section{Conclusions}%

		In this paper we showed that adaptive proximal gradient methods are universal, in the sense that they converge under mere local H\"older gradient continuity of \emph{any} order.
		This is achieved through a unified analysis of \adaPG{} that encapsulates existing methods for different values of the parameter \(\rp \in[1,2]\).
		Sequential convergence along with an \(O(\nicefrac1{K^{\q}})\) rate for the cost is established.
		Remarkably, the analysis and implementation of the algorithm does not require a-priori knowledge of the order \(\q\in(0,1]\) of H\"older continuity.
		In other words, \adaPG{} and its analysis automatically adapts to the best possible choice of \(\q\).

		The validity of some of the auxiliary results for \(\q=0\) is an encouraging indication that the algorithm could potentially cope with plain real valuedness of \(f\), waiving differentiability assumptions.
		Whether this is really the case remains an interesting open question.

		Moreover, our experiments demonstrate that \adaPG{} consistently outperforms \NUPG{} on a diverse collection of challening convex optimization problems with both locally and globally H\"older smooth costs.
		In many cases we observe that \adaPG{} performs better than the accelerated algorithms \FNUPG{} and \ACFGM{}.
		We conjecture that \adaPG{} exploits a hidden H\"older growth that accelerated algorithms cannot take advantage from (as is known for the classical Euclidean case under metric subregularity).
		In future work we aim to extend our analysis to nonconvex and stochastic settings.

\ifarxiv\else

	\section*{Acknowledgements}
		\TheFunding

	\section*{Impact statement}

		This paper presents work whose goal is to advance the field of Machine Learning. There are many potential societal consequences of our work, none which we feel must be specifically highlighted here.

	\bibliographystyle{icml2024}
	\bibliography{Bibliography.bib}
\fi

\ifarxiv
	\begin{center}
		\bigskip
		\phantomsection
		\LARGE\bfseries
		Appendix
	\end{center}
	\addcontentsline{toc}{section}{Appendix}
\else
	\newpage
\fi
	\appendix\InAppendixtrue

	\ifarxiv\else\onecolumn\fi

	\proofsection{sec:Holder}
		\ifarxiv\else\vspace{-\baselineskip}\fi

		\begin{appendixproof}{thm:Lq}
			The proof of assertion \ref{thm:Lq:innprod:leq} follows by \ifarxiv applying \fi the Cauchy-Schwarz inequality.
			For assertion \ref{thm:Lq:f:leq} when \(\q>0\) we refer the reader to \cite[Lem. 1]{yashtini2016global}; although the reference assumes global H\"older continuity, the arguments therein only use H\"older continuity of \(\nabla f\) on the segment \([x,y]\).
			For the case \(\q=0\), for any \(x,y\in E\) and \(\nabla h(x)\in\partial h(x)\), \(\nabla h(y)\in\partial h(y)\), we have
			\begin{align*}
				h(y)
			\leq{} &
				h(x)+\innprod{\nabla h(y)}{y-x}
			\\
			={} &
				h(x)
				+
				\innprod{\nabla h(x)}{y-x}
				+
				\innprod{\nabla h(y)-\nabla h(x)}{y-x}
			\\
			\leq{} &
				h(x)
				+
				\innprod{\nabla h(x)}{y-x}
				+
				\|\nabla h(y)-\nabla h(x)\|\|y-x\|
			\\
			\leq{} &
				h(x)+\innprod{\nabla h(x)}{y-x}+{\def\q{0}\L_E}\|y-x\|,
			\end{align*}
			where the first inequality follows from convexity of \(h\).

			We now turn to the last two claims, and thus restrict to the case \(\q>0\).
			We will only show assertion \ref{thm:Lq:f:geq}, as \ref{thm:Lq:innprod:geq} in turn follows by exchanging the roles of \(x\) and \(y\) and summing the corresponding inequalities.
			The proof follows along the lines of \cite[Thm. 5.8(iii)]{beck2017first} and is included for completeness to highlight the need of the enlarged set \(\overline E\).
			We henceforth fix \(x,y\in E\subseteq\overline E\).
			Since $\nabla f$ is \(\q\)-H\"older continuous in $\overline E$ with modulus \(\L_{\overline E}\), it follows from assertion \ref{thm:Lq:f:leq} that
			\begin{equation}\label{eq:holder_ineq}
				f(z)
			\leq
				f(y)
				+
				\innprod{\nabla f(y)}{z-y}
				+
				\tfrac{\L_{\overline E}}{1+\q}\|z-y\|^{1+\q}
			\quad
				\forall z \in \overline E.
			\end{equation}
			Let \(l_x(y) \coloneqq f(y) - f(x) - \innprod{\nabla f(x)}{y-x}\), and note that \(l_x\) is a convex function with \(\nabla l_x(y)=\nabla f(y)-\nabla f(x)\).
			For any \(z \in \overline E\), we have
			\begin{align*}
				l_x(z)
			={} &
				f(z) - f(x) - \innprod{\nabla f(x)}{z-x}
			\\
			\overrel[\leq]{\eqref{eq:holder_ineq}}{} &
				f(y)
				+
				\innprod{\nabla f(y)}{z-y}
				+
				\tfrac{\L_{\overline E}}{1+\q}\|z-y\|^{1+\q}
				-
				f(x)
				-
				\innprod{\nabla f(x)}{z-x}
			\\
			={} &
				f(y) - f(x) - \innprod{\nabla f(x)}{y-x}
				+
				\innprod{\nabla f(y) - \nabla f(x)}{z-y}
				+
				\tfrac{\L_{\overline E}}{1+\q}\|z-y\|^{1+\q}
			\\
			={} &
				l_x(y)
				+
				\innprod{\nabla l_x(y)}{z-y}
				+
				\tfrac{\L_{\overline E}}{1+\q}\|z-y\|^{1+\q}.
			\end{align*}
			Noticing that $\nabla l_x(x) = 0$, it follows from convexity that $x$ is a global minimizer of $l_x$, hence that \(\min l_x=l_x(x)=0\).
			Let us denote \(v\coloneqq\frac{1}{\|\nabla l_x(y)\|}\nabla l_x(y)\) and define
			\(
				z = y -\|\nabla l_x(y)\|^{1/\q}\L_{\overline E}^{-1/\q}v
			\).
			Note that
			\[
				\|z-y\|
			=
				\Bigl(
					\tfrac{\|\nabla l_x(y)\|}{\L_{\overline E}}
				\Bigr)^{\frac{1}{\q}}
			=
				\Bigl(
					\tfrac{\|\nabla f(y)-\nabla f(x)\|}{\L_{\overline E}}
				\Bigr)^{\frac{1}{\q}}
			\leq
				\|y-x\|,
			\]
			and in particular \(z\in E+\cball{y}{\|y-x\|}\subseteq\overline E\).
			From the previous inequality we get
			\begin{align*}
				0
			=
				\min l_x
			=
				l_x(x)
			\leq{} &
				l_x\bigl(
					y - \|\nabla l_x(y)\|^{1/\q}{\L_{\overline E}^{-1/\q}}v
				\bigr).
			\\
			\overrel[\leq]{\eqref{eq:holder_ineq}}{} &
				l_x(y)
				-
				\|\nabla l_x(y)\|^{1/\q}
				\L_{\overline E}^{-1/\q}
				\innprod{\nabla l_x(y)}{v}
				+
				\tfrac{\L_{\overline E}}{1+\q}
				\tfrac{1}{\L_{\overline E}^{1+1/\q}}
				\|\nabla l_x(y)\|^{1+\frac{1}{\q}}
			\\
			={} &
				f(y) - f(x) - \innprod{\nabla f(x)}{y-x}
				-
				\tfrac{\q}{\q+1}
				\tfrac{1}{\L_{\overline E}^{1/\q}}
				\|\nabla f(x) - \nabla f(y)\|^{1+1/\q},
			\end{align*}
			as claimed.
		\end{appendixproof}

		\begin{appendixproof}{thm:Hk}%
			Expanding the squares yields
			{\ifarxiv\mathtight[0.78]\fi
			\begin{align*}
				\|\Hk(x^k)-\Hk(x^{k-1})\|^2
			={} &
				\|x^k-x^{k-1}\|^2
				+
				\gamk^2\|\nabla f(x^k)-\nabla f(x^{k-1})\|^2
				+
				2\gamk
				\innprod{x^k-x^{k-1}}{\nabla f(x^k)-\nabla f(x^{k-1})}
			\\
			={} &
				\|x^k-x^{k-1}\|^2
				+
				\gamk^2
				\Lk^2\|x^k-x^{k-1}\|^{2\q}
				+
				2\gamk
				\lk\|x^k-x^{k-1}\|^{1+\q}.
			\end{align*}}%
			From the identity \(\gamk=\lamk\|x^k-x^{k-1}\|^{1-\q}>0\), the claimed expression follows.
		\end{appendixproof}

	\proofsection{sec:lemmas}
		\ifarxiv\else\vspace{-\baselineskip}\fi

		\begin{appendixproof}{thm:FNE}

			Recall the subgradient characterization
			\begin{equation}
			\label{eq:Dg}
				\nabla*g(x^{k+1})
			\coloneqq
				\tfrac{x^k-x^{k+1}}{\gamk*}-\nabla f(x^k)
			\in
				\partial g(x^{k+1}).
			\end{equation}
			We have
			\begin{align*}
				\innprod{\tfrac{\Hk(x^{k-1})-\Hk(x^k)}{\gamk}}{x^k-x^{k+1}}
			={} &
				\innprod{\tfrac{\Hk(x^{k-1})-x^k}{\gamk}+ \nabla f(x^k)}{x^k-x^{k+1}}
			\\
			={} &
				\tfrac{1}{\gamk*}
				\|x^{k+1} - x^k\|^2
				+
				\innprod{
					x^k - x^{k+1}
				}{
					\underbracket*[0.5pt]{
						\tfrac{\Hk(x^{k-1})-x^k}{\gamk\vphantom{\gamk*}}
					}_{\in\partial g(x^k)}
					-
					\underbracket*[0.5pt]{
						\tfrac{\Hk*(x^k)-x^{k+1}}{\gamk*}
					}_{\in\partial g(x^{k+1})}
					}
			\\
			\geq{}&
				\tfrac{1}{\gamk*}
				\|x^{k+1} - x^k\|^2,
			\end{align*}
			owing to monotonicity of \(\partial g\).
			Invoking the Cauchy-Schwarz inequality completes the proof.
		\end{appendixproof}

		\begin{appendixproof}{lemma:main:inequality}
		    From the convexity inequality for \(g\) at \(x^{k+1}\) with subgradient \(\nabla*g(x^{k+1})\) as in \eqref{eq:Dg}, we have that for any \(x\in \dom g\)
		    \begin{align}
		    \nonumber
		        0
		    \leq{} &
		        g(x)-g(x^{k+1})
		        +
		        \innprod{\nabla f(x^k)}{x-x^{k+1}}
		        -
		        \tfrac{1}{\gamk*}\innprod{x^k-x^{k+1}}{x-x^{k+1}}
		    \\
		    ={} &
		        g(x)-g(x^{k+1})
		        +
		        \innprod{\nabla f(x^k)}{x-x^{k+1}}
		        +
		        \tfrac{1}{2\gamk*}
		        \left(
		            \|x^k-x\|^2
		            -
		            \|x^{k+1}-x\|^2
		            -
		            \|x^k-x^{k+1}\|^2
		        \right).
		    \label{eq:g:xx+}
		    \end{align}
		    On the other hand, using \(\nabla*g(x^k)\in\partial g(x^k)\) yields that
		    \begin{equation}\label{eq:gx+x}
		        0
		    \leq
		        g(x^{k+1})-g(x^k)
		        +
		        \innprod{\nabla f(x^{k-1})}{x^{k+1}-x^k}
		        -
		        \tfrac{1}{\gamk}\innprod{x^{k-1}-x^k}{x^{k+1}-x^k}.
		    \end{equation}
		    We now rewrite the inner product in \eqref{eq:g:xx+} is a way that allows us to exploit the above inequality:
		    \begin{align*}
		        \innprod{\nabla f(x^k)}{x-x^{k+1}}
		    ={} &
		        \innprod{\nabla f(x^k)}{x-x^k}
		        +
		        \innprod{\nabla f(x^k)}{x^k-x^{k+1}}
		    \\
		    \leq{} &
		        f(x)-f(x^k)
		        +
		        \innprod{\nabla f(x^k)}{x^k-x^{k+1}}.
		    \end{align*}
		    The last inner product can be controlled using \eqref{eq:gx+x}:
		    \begin{align*}
		        \innprod{\nabla f(x^k)}{x^k-x^{k+1}}
		    ={} &
		        \innprod{\nabla f(x^k) - \nabla f(x^{k-1})}{x^k-x^{k+1}}
		        +
		        \innprod{\nabla f(x^{k-1})}{x^k-x^{k+1}}
		    \\
		    \leq{} &
		        \innprod{\nabla f(x^k) - \nabla f(x^{k-1})}{x^k-x^{k+1}}
		        +
		        g(x^{k+1})-g(x^k)
		        -
		        \tfrac{1}{\gamk}
		        \innprod{x^{k-1}-x^k}{x^{k+1}-x^k}
		    \\
		    ={} &
		        g(x^{k+1})-g(x^k)
		        +
		        \tfrac{1}{\gamk}
		        \innprod{\Hk(x^{k-1})-\Hk(x^k)}{x^k-x^{k+1}}
		    \\
		        \dueto{(\cref{thm:FNE})}
		        \quad
		    \leq{} &
		        g(x^{k+1})-g(x^k)
		        +
		        \tfrac{\rhok*}{\gamk}\|\Hk(x^{k-1})-\Hk(x^k)\|^2
		    \end{align*}
		    Combine this with the prior inequality and \eqref{eq:g:xx+} to obtain
		    \begin{align}
		            0
		    \leq{}&
		        \gamk*
		        (\varphi(x)-\varphi(x^k))
		        +
		        \rhok*^2\|\Hk(x^{k-1})-\Hk(x^k)\|^2
		        +
		        \tfrac{1}{2}
		        \left(
		            \|x^k-x\|^2
		            -
		            \|x^{k+1}-x\|^2
		            -
		            \|x^k-x^{k+1}\|^2
		        \right)
		        \\
		        ={}&
		        \gamk*
		        (\varphi(x)-\varphi(x^k))
		        +
		        \rhok*^2M_k
		        \|x^k - x^{k-1}\|^2
		        +
		        \tfrac{1}{2}
		            \|x^k-x\|^2
		            -
		        \tfrac{1}{2}
		            \|x^{k+1}-x\|^2
		            -
		        \tfrac{1}{2}
		            \|x^k-x^{k+1}\|^2,
		    \end{align}
		    where \cref{thm:Hk} was used in the last equality.
		    Finally, recalling the definition of \(\nabla*g(x^k)\) as in \eqref{eq:Dg} we have
		    \[
		        \nabla f(x^k)
		        +
		        \nabla*g(x^k)
		    =
		        \tfrac{1}{\gamk}
		        \bigl(\Hk(x^{k-1})-\Hk(x^k)\bigr)
		    =
		        \tfrac{x^{k-1}-x^k}{\gamk}+ \nabla f(x^k)-\nabla f(x^{k-1})
		    \in
		        \partial \varphi(x^k).
		    \]
		    Therefore, for any \(\vartheta_{k+1}\geq0\)
		    \begin{align*}
		        0
		    \leq{} &
		        \gamk*\vartheta_{k+1}
		        \left(
		            \varphi(x^{k-1})-\varphi(x^k)
		            -
		            \tfrac{1}{\gamk}
		            \innprod{\Hk(x^{k-1}) -\Hk(x^k)}{x^{k-1} - x^k}
		        \right)
		    \\
		    ={} &
		        \gamk*\vartheta_{k+1}
		        \left(
		            \varphi(x^{k-1})- \varphi(x^k)
		            -
		                \tfrac{1-\gamk\ell_k}{\gamk}
		                \|x^{k-1} - x^k\|^2
		        \right).
		    \end{align*}
		    Combining the last two inequalities and letting \(P_k(x) \coloneqq \varphi(x^k) - \varphi(x)\) yields
		    \begin{align*}
		            &{}
		            \tfrac{1}{2}\|x^{k+1}-x\|^2
		            +
		            \gamk* (1 + \vartheta_{k+1})P_k(x)
		            +
		            \tfrac{1}{2}\|x^k-x^{k+1}\|^2
		            \\
		    \leq{}&
		        \tfrac{1}{2}\|x^k-x\|^2
		        +
		        \gamk*\vartheta_{k+1}
		        P_{k-1}(x)
		       -
		        \rhok*
		        \left(%
		            \vartheta_{k+1}(1-\gamk\ell_k)
		            - \rhok* M_k^2
		        \right)
		              \|x^{k-1} - x^k\|^2
		    \end{align*}

		    Letting \(\vartheta_k = \rp \rhok\) for all \(k\), and setting \(x = x^\star\in \argmin \varphi\) establishes the claimed inequality.
		\end{appendixproof}

		\vspace*{-2\baselineskip}%
		\begin{appendixproof}{thm:Lyapunov}~
			\begin{proofitemize}[topsep=0pt]
			\item \ref{thm:descent})
				The update rule for \(\gamk\) ensures that the coefficients of \(\|x^k-x^{k-1}\|^2\) and \(P_{k-1}\) in \eqref{eq:SD} are negative, and that therefore \(\seq{\U_k(x^\star)}\) is decreasing.
				Since \(\U_k(x^\star)\geq0\), the sequence converges to a finite (positive) value.

			\item \ref{thm:bounded})
				Boundedness follows by observing that \(\frac{1}{2}\|x^k-x^\star\|^2\leq\U_k(x^\star)\leq\U_0(x^\star)\) for all $k \geq 0$, where the last inequality owes to the previous assertion.
				In what follows, we let \(\L_{\Omega}<\infty\) be a \(\q\)-H\"older modulus for \(\nabla f\) on a convex and compact set \(\Omega\) that contains all the iterates \(x^k\).
				In particular, \(\lk\leq\Lk\leq\L_{\Omega}\) holds for every \(k\).
				Suppose that \(x_\infty\) and \(x_\infty'\) are two optimal limit points, and observe that
				\[
					\U_k(x_\infty)-\U_k(x_\infty')
				=
					\tfrac{1}2\|x_\infty\|^2
					+
					\tfrac{1}2\|x_\infty'\|^2
					+
					\innprod{x^k}{x_\infty'-x_\infty}.
				\]
				Since both \(\seq{\U_k(x_\infty)}\) and \(\seq{\U_k(x_\infty')}\) are convergent, by taking the limit along the subsequences converging to \(x_\infty\) and \(x_\infty'\) we obtain
				\(
					\innprod{x_\infty}{x_\infty'-x_\infty}
				=
					\innprod{x_\infty'}{x_\infty'-x_\infty}
				\),
				which after rearranging yields \(\|x_\infty-x_\infty'\|^2=0\).

			\item \ref{thm:sumgamk})
				Since \(\gamk(1+\rp\rhok-\rp\rhok*^2)\geq0\) because of the update rule of \(\gamk*\), and with \(P_k^{\rm min}\coloneqq\min_{0\leq i\leq k}P_i\) denoting the best-so-far cost at iteration \(k\), a telescoping argument on \eqref{eq:SD} yields that
				\begin{align*}
					P_K^{\rm min}
					\sum_{k=1}^K\gamk(1+\rp\rhok-\rp\rhok*^2)
				\leq{} &
					\sum_{k=1}^K\gamk(1+\rp\rhok-\rp\rhok*^2)P_{k-1}
				\\
				\leq{} &
					\U_1(x^\star)-\U_{K+1}(x^\star)
				\leq
					\U_1(x^\star)-\gam_{K+1}(1+\rp \rho_{K+1})P_K^{\rm min},
				\numberthis\label{eq:sumPk}
				\end{align*}
				hence that
				\begin{align*}
					P_K^{\rm min}
				\leq{} &
					\frac{
						\U_1(x^\star)
					}{
						\sum_{k=1}^{K+1}\gamk(1+\rp\rhok-\rp\rhok*^2)
						+
						\gamk*(1+\rp\rho_{K+1})
					}
				\\
				={} &
					\frac{
						\U_1(x^\star)
					}{
						\sum_{k=1}^K(\gamk+\rp\gamk\rhok-\rp\gamk*\rhok*)
						+
						\gamk*
						+
						\rp\gamk*\rho_{K+1}
					}
				\\
				={} &
					\frac{
						\U_1(x^\star)
					}{
						\rp \gam_1\rho_1+ \sum_{k=1}^{K+1}\gamk
					}.
				\end{align*}
				Further using the fact that \(\U_1(x^\star)\leq\U_0(x^\star)\) by \cref{lemma:main:inequality} completes the proof.
			\qedhere
			\end{proofitemize}
		\end{appendixproof}

\begin{WhereToPutThis}
	\proofsection{sec:convergence}
		\ifarxiv\else\vspace{-\baselineskip}\fi

		\begin{appendixproof}{thm:lammin}

				Owing to \(\rhok* \leq \sqrt{\nicefrac1{\rp}+ \rhok}\) as ensured in \eqref{state:PG:gamk*}, it can be verified with a trivial induction argument that
		\begin{equation}\label{eq:rhomax}
			\rhok
		\leq
			\rho_{\rm max}
		\coloneqq
			\max\set{\tfrac{1}{2}(1+\sqrt{1+\nicefrac{4}{\rp}}), \rho_0}
		\quad
			\text{for all } k \geq 0.
		\end{equation}

		If \(k\in K_2\), then \(\gamk*\) coincides with the second update in \eqref{eq:gamk*}, and thus
		\begin{equation}\label{eq:2active:lb}
			\rho_{\rm max}
		\geq
			\rhok*
		=
			\frac{1}{
				\sqrt{2\left[\lamk^2\Lk^2-(2-\rp)\lamk\lk + 1- \rp\right]_+}
			}
		\geq
			\frac{1}{\sqrt2\lamk\Lk},
		\end{equation}
		completing the proof.
		\end{appendixproof}

			In our convergence analysis we will need the following lemma that extends \cite[Lem. B.2]{latafat2023convergence} by allowing a vanishing stepsize.
	As a result it is only \(\gamk*\) times the cost that can be ensured to converge to zero, which will nevertheless prove sufficient for our convergence analysis in the proof of \cref{thm:convergence}.

	\begin{lemma}\label{thm:Pkto0}%
		Suppose that a sequence \(\seq{x^k}\) converges to an optimal point \(x^\star\in\argmin\varphi\), and for every \(k\) let \(\bar x^k\coloneqq\prox_{\gamk* g}(x^k-\gamk*\nabla f(x^k))\) with \(\seq{\gamk}\subset\R_{++}\) bounded.
		Then, \(\seq{\bar x^k}\) too converges to \(x^\star\) and \(\seq{\gamk*(\varphi(\bar x^k) - \min \varphi)}\to 0\).
	\end{lemma}
	\begin{proof}
		By nonexpansiveness of the proximal mapping
		\[
			\|\bar x^k - x^\star\|
		\leq
			\|x^k - x^\star - \gamk*(\nabla f(x^k) - \nabla f(x^\star))\|
		\leq
			\|x^k - x^\star\|
			+
			\gamk*\|\nabla f(x^k) - \nabla f(x^\star)\| \to 0,
		\]
		where we used the fact that \(x^\star = \prox_{\gamk* g}(x^\star - \gamk* \nabla f(x^\star))\) for any \(\gamk*>0\) in the first inequality, and boundedness of \(\gamk*\) in the last implication.
		Moreover, for every \(k\in\N\) one has
		\[
			\gamk*(\varphi(\bar x^k) - \min \varphi)
		=
			\gamk* (f(\bar x^k)+g(\bar x^k) - \min \varphi)
		\leq
			\gamk*(f(\bar x^k)-f(x^\star))
			-
			\langle x^k - \gamk* \nabla f(x^k) - \bar x^k, x^\star - \bar x^k \rangle,
		\]
		where in the inequality we used the subgradient characterization of the proximal mapping.
		The inner product vanishes since both \(x^k\) and \(\bar x^k\) converge to \(x^\star\), and the claim follows by  continuity of \(f\) and lower semicontinuity of \(\varphi\).
	\end{proof}

		\ifarxiv\else\vspace{-\baselineskip}\fi

		\begin{appendixproof}{thm:convergence}%

		We first show two intermediate claims.
		\def\currentlabel{thm:convergence}%
		\begin{claims}
		\item \label{claim:infP}%
			{\em If \(\q>0\), then \(\inf_{k\in\N}P_k=0\), and in particular \(\seq{x^k}\) admits a (unique) optimal limit point.}

			If \(\sup_{k\in\N}\gamk=\infty\), then we know from \cref{thm:sumgamk} that \(\liminf_{k\to\infty}P_k=0\).
			Suppose instead that \(\seq{\gamk}\) is bounded.
			Then, the set \(K_2\) as in \eqref{eq:K2} must be infinite.
			Let \(\L_\Omega\) be a \(\q\)-H\"older modulus for \(\nabla f\) on a compact convex set \(\Omega\) that contains all the iterates \(x^k\), ensured to exist by \cref{thm:bounded}.
			Since \(\Lk\leq\L_{\Omega}\), it follows from \cref{thm:lammin} that
			\begin{equation}\label{eq:lammin}
				\lamk
			\geq
				\lam_{\rm min}
			\coloneqq
				\frac{1}{\sqrt2\L_{\Omega}\rho_{\rm max}}
			\quad
				\forall k\in K_2,
			\end{equation}
			hence from \eqref{eq:sumPk} that
			\[
				\sum_{k\in K_2}\|x^k-x^{k-1}\|^{1-\q}(1+\rp\rhok-\rp\rhok*^2)P_{k-1}
			<
				\infty.
			\]
			Noticing that \(1+\rp\rhok-\rp\rhok*^2=0\) for \(k\notin K_2\), necessarily \(1+\rp\rhok-\rp\rhok*^2\not\to0\) as \(K_2\ni k\to\infty\) (or, equivalently, as \(k\to\infty\)), for otherwise \(\liminf_{k\to \infty} \rho_k >1\) and thus \(\gamk\nearrow\infty\).
			Therefore, there exists an infinite set \(\tilde K_2\subseteq K_2\) such that \(1+\rp\rhok-\rp\rhok*^2\geq\varepsilon>0\) for all \(k\in\tilde K_2\), implying that
			\[
				\sum_{k\in\tilde K_2}\|x^k-x^{k-1}\|^{1-\q}P_{k-1}
			<
				\infty.
			\]
			Thus, \(\lim_{\tilde K_2\ni k\to\infty}\|x^k-x^{k-1}\|=0\) (or \(\liminf_{k\in\tilde K_2}P_{k-1}=0\), in which case there is nothing to show).
			For any \(x^\star\in\argmin\varphi\) we thus have
			\begin{align*}
				0
			\leq
				P_k
			=
				\varphi(x^k)-\min\varphi
			\leq{} &
				\innprod{x^k-x^\star}{\tfrac{x^{k-1}-x^k}{\gamk}-\bigl(\nabla f(x^{k-1})-\nabla f(x^k)\bigr)}
			\\
			\numberthis\label{eq:Pkbound}
			\leq{} &
				\|x^k-x^\star\|
				\left(
					\tfrac{1}{\gamk}
					\|x^{k-1}-x^k\|
					+
					\|\nabla f(x^{k-1})-\nabla f(x^k)\|
				\right)
			\\
			={} &
				\|x^k-x^\star\|
				\left(
					\tfrac{1}{\lamk\|x^{k-1}-x^k\|^{1-\q}}
					\|x^{k-1}-x^k\|
					+
					\Lk
					\|x^{k-1}-x^k\|^{\q}
				\right)
			\\
			\leq{} &
				\|x^k-x^\star\|
				\|x^{k-1}-x^k\|^{\q}
				\left(
					\tfrac{1}{\lam_{\rm min}}
					+
					\L_\Omega
				\right)
			\quad
				\forall k\in\tilde K_2.
			\end{align*}
			Since \(\q>0\), by taking the limit as \(\tilde K_2\ni k\to\infty\) we obtain that \(\lim_{\tilde K_2\ni k\to\infty}P_k=0\).

		\item
			{\em If \(\q>0\) and \(\seq{\gamk}\) is bounded, then \(\seq{x^k}\) converges to a solution.}

			Suppose first that \(\seq{\gamk}\) is bounded.
			Consider a subsequence \(\seq{x^k}[k\in K]\) such that \(\lim_{K\ni k\to\infty}P_k=0\), which exists and converges to a solution \(x^\star\) by \cref{claim:infP}.
			Since \(\seq{\gamk}\) is bounded, in light of \cref{thm:Pkto0} also \(x^{k+1}\to x^\star\) and, in turn, \(x^{k+2}\to x^\star\) as well.
			Then,
			\[
				\U_{k+1}(x^\star)
			=
				\tfrac{1}2\|x^{k+1}-x^\star\|^2
				+
				\tfrac{1}2\|x^{k+1}-x^k\|^2
				+
				\gamk\left(1+\rp\rhok*\right)P_k
			\to
				0
			\quad
				\text{as }K\ni k\to\infty,
			\]
			and thus \(\frac{1}2\|x^k-x^\star\|\leq\U_k(x^\star)\to0\) as \(k\to\infty\), since the entire sequence \(\seq{\U_k(x^\star)}\) is convergent.
		\end{claims}

		To conclude the proof ot the theorem, it remains to show that also in case \(\seq{\gamk}\) is unbounded the sequence \(\seq{x^k}\) converges to a solution.
		To this end, let us suppose now that \(\gamk\) is not bounded.
		This case requires requires a few more technical steps, which can nevertheless almost verbatim be adapted from the proof of \cite[Thm. 2.4(ii)]{latafat2023convergence}. See also \cite[Thm. 2.3(iii)]{latafat2023adaptive} for an alternative argument;
		we emphasize that the difference with both aforementioned works is that the stepsize sequence is not guaranteed to be bounded away from zero.

		We start by observing that \cref{claim:infP} and \cref{thm:bounded} ensure that an optimal limit point \(x^\star\in\argmin\varphi\) exists.
		It then suffices to show that \(\U_k(x^\star)\) converges to zero.
		To arrive to a contradiction, suppose that this is not the case, that is, that
		\(
			U \coloneqq \lim_{k \to \infty}\U_k(x^\star)>0
		\).
		We shall henceforth proceed by intermediate claims that follow from this condition, eventually arriving to a contradictory conclusion.

		\begin{claims*}
		\item \label{claim*:PG:gamk*}%
			{\em For any \(K\subseteq\N\),
				\(
					\lim_{K\ni k\to\infty} x^k = x^\star
				\)
				holds iff
				\(
					\lim_{K\ni k\to\infty} \gamk* = \infty
				\).
			}%

			The implication ``\(\Leftarrow\)'' follows from
			\[
				\gamk P_{k-1} \leq \U_k(x^\star) < \infty
			\]
			since \(\seq{x^k}\) is bounded and \(x^\star\) is its unique optimal limit point.

			Suppose now that \(\seq{x^k}[k\in K]\to x^\star\).
			To arrive to a contradiction, up to possibly extracting another subsequence suppose that \(\seq{\gamk*}[k\in K]\to\bar\gamma \in [0, \infty)\).
			Then, it follows from \cref{thm:Pkto0} that \(\seq{x^{k+1}}[k\in K]\to x^\star\) and \(\seq{\gamk*P_{k+1}}[k\in K]\to0\).
			As shown in \eqref{eq:rhomax}
			\begin{equation}\label{eq:rhomax:2}
				\rhok \leq \rho_{\rm max} \quad \text{for all } k \geq 0
				\end{equation}
			which in turn implies \(\seq{\gam_{k+2} P_{k+1}}[k\in K] \to 0\) and that \(\seq{\gamma_{k+2}}[k\in K]\) is also bounded, we may iterate and infer that also \(\seq{x^{k+2}}[k\in K]\) converges to \(x^\star\).
			Recalling the definition of \(\U_k\) in \eqref{eq:Uk},
			\begin{align*}
				\U_{k+2}(x^\star)
			\coloneqq{} &
				\tfrac12\|x^{k+2}-x^\star\|^2
				+
				\tfrac12\|x^{k+2}-x^{k+1}\|^2
				+
				\gam_{k+2}(1+\rp\rho_{k+2})P_{k+1}
			\to 0,
			\end{align*}
			contradicting \(U=\lim_{K\ni k\to\infty}\U_{k+2}(x^\star)>0\).

		\item
			{\em Suppose that \(\seq{x^k}[k\in K]\to x^\star\); then also \(\seq{x^{k-1}}[k\in K]\to x_\star\).}

			It follows from the previous claim that \(\lim_{K\ni k\to\infty}\gamk*=\infty\).
			Because of \eqref{eq:rhomax:2}, one must also have \(\lim_{K\ni k\to\infty}\gamk=\infty\).
			Invoking again the previous claim, by the arbitrarity of the index set \(K\) the assertion follows.

		\item \label{claim*:gamkLk:infty}
			{\em
				Suppose that \(\seq{x^k}[k\in K]\to x^\star\); then \(\seq{\gam_{k-1} L_{k-1}}[k\in K]\to\infty\){}  and \(\seq{\rhok}[k\in K]\to 0\).
			}%

			Using the previous claim twice, \(x^{k-1},x^{k-2}\to x^\star\) as \(K\ni k\to\infty\).
			In particular
			\begin{equation}\label{eq:shifted:res}
				\lim_{K\ni k\to\infty}\|x^{k-1}-x^{k-2}\|^2=0.
			\end{equation}
			From the expression \eqref{eq:Uk} of \(\U_k\)  we then have
			\begin{equation}\label{eq:gamkPk:U}
				\lim_{K\ni k\to\infty}\gamk(1+\rp\rhok)P_{k-1}
			=
				U,
			\end{equation}
			where we remind that by contradiction assumption \(U\coloneqq\lim_{k\to\infty}\U_k(x^\star)>0\).
			Denoting \(C\coloneqq\rho_{\rm max}(1+\rp\rho_{\rm max})\), we have
			\begin{align*}
				\gamma_{k-1} P_{k-1}
			\leq{} &
				\|x^{k-1}-x^\star\|
				\left(
					\|x^{k-1}-x^{k-2}\|
					+
					\gam_{k-1}\|\nabla f(x^{k-1})-\nabla f(x^{k-2})\|
				\right)
			\\
			={} &
				\|x^{k-1}-x^\star\|
				\left(
					\|x^{k-1}-x^{k-2}\|
					+
					\gam_{k-1}L_{k-1}\|x^{k-1}-x^{k-2}\|^{\q}
				\right)
			\end{align*}
			for every \(k\). Then, by \eqref{eq:gamkPk:U}
			\begin{align*}
				0
			<
				U
			={} &
				\lim_{K\ni k\to\infty}
				\gamk(1+\rp\rhok)P_{k-1}
			={}
				\liminf_{K\ni k\to\infty}
				\rhok(1+\rp\rhok)\gamma_{k-1}P_{k-1}
			\\
			\leq{} &
				\rho_{\rm max}(1+\rp\rho_{\rm max})
				\liminf_{K\ni k\to\infty}
				\|x^{k-1}-x^\star\|
				\left(
					\|x^{k-1}-x^{k-2}\|
					+
					\gam_{k-1}L_{k-1}\|x^{k-1}-x^{k-2}\|^{\q}
				\right)
			\end{align*}
			which yields the first claim owing to \eqref{eq:shifted:res} and \(\q>0\).

			This along with the update rule \eqref{eq:gamk*} implies
			\begin{align*}
				\rhok \leq
				\frac{
					1
				}{
					2\left[\gam_{k-1}^2L_{k-1}^2-(2-\rp)\gam_{k-1}\ell_{k-1} + 1-\rp\right]
				} \to 0,
			\end{align*}
			as claimed.
		\end{claims*}
		Having shown the above claims, the proof is concluded as in \cite[Thm. 2.4(ii)]{latafat2023convergence} by constructing a specific unbounded stepsize sequence and using claims \ref{claim*:PG:gamk*} and \ref{claim*:gamkLk:infty} to obtain the sought contradiction.
		\end{appendixproof}

	\vspace*{-2\baselineskip}%
	\begin{appendixproof}{thm:sublinear}
		The existence of \(\Omega\) as in the statement follows from boundedness of \(\seq{x^k}\), see \cref{thm:bounded}.
		We proceed by intermediate claims.
		\def\currentlabel{thm:sublinear}%
		\begin{claims}
		\item \label{thm:sublinear:Pk}%
			{\em If \(\lamk\leq\frac{1}{\L_\Omega}\), then \(P_k \leq P_{k-1}\).}

			Let \(\nabla*\varphi(x^k)\coloneqq\nabla f(x^k)+\nabla*g(x^k)\), where \(\nabla*g\) is as in \eqref{eq:Dg}.
			Then, \(\nabla*\varphi(x^k)\in\partial\varphi(x^k)\) and thus
			\begin{align*}
				\varphi(x^{k-1})
			\geq{} &
				\varphi(x^k)
				+
				\innprod{\nabla*\varphi(x^k)}{x^{k-1}-x^k}
			\\
			={} &
				\varphi(x^k)
				+
				\tfrac{1}{\gamk}
				\innprod{\Hk(x^{k-1})-\Hk(x^k)}{x^{k-1}-x^k}
			\\
			={} &
				\varphi(x^k)
				+
				\tfrac{1}{\gamk}\|x^k-x^{k-1}\|^2
				-
				\Lk\|x^{k-1}-x^k\|^{1+\q}
			\\
			={} &
				\varphi(x^k)
				+
				\bigl(
					\tfrac{1}{\lamk}
					-
					\Lk
				\bigr)
				\|x^{k-1}-x^k\|^{1+\q},
			\end{align*}
			establishing the claim.

			We next aim at establishing a lower bound on the stepsize sequence in terms of \(P_k^{\frac{1-\q}{\q}}\).
			To simplify the exposition, we now fix \(x^\star\in\argmin\varphi\) and denote
			\begin{equation}\label{eq:Cnu}
				\tilde C_{\q}^\rp(\nu)
			\coloneqq
				\sqrt{2\U_1(x^\star)}
				\bigl(\tfrac{1}{\nu}+\L_{\Omega}\bigr).
			\end{equation}

		\item \label{thm:sublinear:Pk<Dk}%
			{\em
				For every \(k\in\N\) it holds that
				\(
					P_k
				\leq
					\tilde C_{\q}^\rp(\lamk)
					\|x^k-x^{k-1}\|^{\q}
				\).
			}

			We begin by observing that
			\[
				\nabla*\varphi(x^k)
			\coloneqq
				\tfrac{1}{\gamk}
				\bigl(\Hk(x^{k-1})-\Hk(x^k)\bigr)
			\in
				\partial\varphi(x^k)
			\]
			owing to \eqref{eq:Dg}.
			Combined with \eqref{eq:hkq:bound} and \eqref{eq:lamk} it follows that
			\[
				\|\nabla* \varphi(x^k)\|
			\leq
				\bigl(\tfrac{1}{\lamk}+\Lk\bigr)
				\|x^k-x^{k-1}\|^{\q}
			\leq
				\bigl(\tfrac{1}{\lamk}+\L_\Omega\bigr)
				\|x^k-x^{k-1}\|^{\q}.
			\]
			Moreover, by convexity,
			\[
				P_k
			=
				\varphi(x^k)
				-
				\min\varphi
			\leq
				\innprod{\nabla*\varphi(x^k)}{x^k-x^\star}
			\leq
				\|\nabla*\varphi(x^k)\|
				\|x^k-x^\star\|
			\leq
				\overbracket*[0.5pt]{
					\sqrt{2\U_1(x^\star)}
					\bigl(\tfrac{1}{\lamk}+\L_\Omega\bigr)
				}^{\tilde C_{\q}^\rp(\lamk)}
				\|x^k-x^{k-1}\|^{\q}
			\]
			as claimed, where the last inequality uses the fact that \(\tfrac12\|x^k-x^\star\|^2\leq\U_k(x^\star)\leq\U_1(x^\star)\).

			We next analyze two possible cases for any iteration index \(k\geq 0\).

		\item \label{claim:Pk:bound:2}%
			{\em
				For any \(k\in K_2\),~
				\(
					\gamk*
				\geq
					\frac{1}{\sqrt{2}\L_\Omega}
					\left(
						\frac{P_k}{\tilde C_{\q}^\rp(\lam_{\rm min})}
					\right)^{\frac{1-\q}{\q}}
				\).
			}%

			Since \(\Lk\leq\L_{\Omega}\), as shown in \cref{thm:lammin}
			\(
				\lamk
			\geq
				\lam_{\rm min}
			=
				\frac{1}{\sqrt2\L_{\Omega}\rho_{\rm max}}
			\)
			holds for every \(k\in K_2\).
			Moreover, by definition of \(K_2\),
			\begin{align}
			\nonumber
				\gamk*
			={} &
				\frac{
					\gamk
				}{
					\sqrt{2\left[\lamk^2\Lk^2-(2-\rp)\lamk\lk + 1- \rp\right]_+}
				}
			\\
			={} &
				\frac{
					\lamk\|x^k-x^{k-1}\|^{1-\q}
				}{
					\sqrt{2\left[\lamk^2\Lk^2-(2-\rp)\lamk\lk + 1- \rp\right]_+}
				}
			\geq
				\frac{\|x^k-x^{k-1}\|^{1-\q}}{\sqrt{2}\Lk}
			\geq
				\frac{\|x^k-x^{k-1}\|^{1-\q}}{\sqrt{2}\L_\Omega}
		\ifarxiv\else
			\quad
				\forall k\in K_2.
		\fi
			\label{eq:lamk*:Dk}
			\end{align}
		\ifarxiv
			holds for all \(k\in K_2\).
		\fi
			By using the lower bound
		\ifarxiv\[\else\(\fi
				\|x^k-x^{k-1}\|^{1-\q}
			\geq
				\Bigl(\frac{P_k}{\tilde C_{\q}^\rp(\lamk)}\Bigr)^{\frac{1-\q}{\q}}
			\geq
				\Bigl(\frac{P_k}{\tilde C_{\q}^\rp(\lam_{\rm min})}\Bigr)^{\frac{1-\q}{\q}}
		\ifarxiv\]\else\)\fi
			in \cref{thm:sublinear:Pk<Dk} raised to the power \(\frac{1-\q}{\q}\) the claim follows.

		\item \label{claim:Pk:bound:1}%
			{\em
				For any \(k\in K_1\),~
				\(
					\gamk*
				\geq
					\begin{ifcases}
						\frac{1}{\sqrt{2\rp}\L_\Omega}
						\bigl(
							\frac{P_k}{\tilde C_{\q}^\rp(\lam_{\rm min})}
						\bigr)^{\frac{1-\q}{\q}}  & \text{(\(K_2\neq\emptyset\) and) } k \geq\min K_2,
					\\[3pt]
						\bigl(
							1 + \tfrac{1}{\rp}
						\bigr)^{\frac{k}{2}} \gam_0
					\otherwise.
					\end{ifcases}
				\)%
			}

			Let
			\[
				K_{2,<k}
			\coloneqq
				K_2\cap\set{0,1,\dots,k-1}
			\]
			denote the (possibly empty) set of all iteration indices up to \(k-1\) such that the first term in \eqref{eq:gamk*} is strictly larger than the second one.

			If \(K_{2,<k} = \emptyset\), then
			\(
				\rho_{t+1} = \sqrt{\nicefrac{1}{\rp}+\rho_t}
			\)
			holds for all \(t\leq k\), which inductively gives
			\(
				\rho_t^2
			\geq
				1 + \tfrac{1}{\rp}
			\)
			for all \(t=1,\dots,K\) (since \(\rho_0 \geq 1\)).
			We then have
			\begin{equation}\label{eq:PG:gam2k}
				\gamk*^2
			=
				\gam_0^2\textstyle\prod_{t=1}^k\rho_{t+1}^2
			\geq
				(1+ \tfrac{1}{\rp})^k\gam_0^2.
			\end{equation}

			Suppose instead that \(K_{2,<k}\neq\emptyset\), and let \(n_{2,k}\) denote its largest element:
			\[
				n_{2,k}
			\coloneqq
				\max K_{2,<k}
			=
				\max\set{i<k}[
					\gam_{i+1}
				<
					\gam_i\sqrt{\tfrac{1}{\rp}+\rho_i}
				].
			\]

			Observe that the update rule \(\rho_{i+1}=\sqrt{\frac{1}{\rp}+\rho_i}\) implies that \(\rho_{i+2}\geq1\) holds whenever \(i,i+1\in K_1\).
			In fact,
			\(
				\rho_{i+1}=\sqrt{\tfrac{1}{\rp}+\rho_i}\geq\tfrac{1}{\sqrt{\rp}}
			\)
			holds for every \(i\in K_1\), in turn implying that
			\[
			\textstyle
				i,i+1\in K_1
			\quad\Rightarrow\quad
				\rho_{i+2}
			\geq
				\sqrt{\frac{1}{\rp}+\sqrt{\frac{1}{\rp}}}
			\geq
				\sqrt{\frac{1}{2}+\sqrt{\frac{1}{2}}}
			>
				1.
			\]

			In particular,
			\begin{equation}\label{eq:prodK1}
			\textstyle
				i,i+1,\dots,j\in K_1
			\quad\Rightarrow\quad
				\prod_{t=i+1}^{j+1}\rho_t
				\geq
				\tfrac{1}{\sqrt{\rp}}
			\end{equation}
			(this being also trivially true for an empty product, since \(\rp\geq1\)).

			We consider two possible subcases:
			\begin{itemize}[label=\(\diamond\),leftmargin=*]
			\item
				First, suppose that the index \(j\coloneqq\max \set{n_{2,k}\leq i \leq k}[\lam_{i}>\tfrac{1}{\L_\Omega}]\) exists.
				Schematically,
				\begin{equation}\label{eq:jscheme}
					\overbracket[0.5pt]{\,
						n_{2,k}
						\vphantom{jk}
					\,}^{\in K_2}
					,
					\overbracket[0.5pt]{\,
						\dots,j
						,
						\underbracket[0.5pt]{\,
							\dots,k
						\,}_{\mathclap{\lam_i\leq\frac{1}{\L_\Omega}}}
					\,}^{\in K_1}
				\quad\text{and}\quad
					\lam_j>\tfrac{1}{\L_\Omega}.
				\end{equation}
				By definition of \(n_{2,k}\), all indices between \(j\) and \(k\) are in \(K_1\), and thus
				\begin{align*}
					\gamk*
				=
					\gam_j
					\prod*_{i=j+1}^{k+1}\rho_i
				\overrel[\geq]{\eqref{eq:prodK1}}{} &
					\tfrac{1}{\sqrt{\rp}}
					\gam_j
				=
					\tfrac{1}{\sqrt{\rp}}
					\lam_j
					\|x^j-x^{j-1}\|^{1-\q}
				\\
				\overrel[>]{\eqref{eq:jscheme}}{} &
					\tfrac{1}{\sqrt{\rp}}
					\tfrac{1}{\L_{\Omega}}
					\|x^j-x^{j-1}\|^{1-\q}
				\overrel[\geq]{\cref{thm:sublinear:Pk<Dk}}[3pt]
					\tfrac{1}{\sqrt{\rp}}
					\tfrac{1}{\L_{\Omega}}
					\Bigl(
						\tfrac{P_j}{\tilde C_{\q}^\rp(\lam_j)}
					\Bigr)^{\frac{1-\q}{\q}}
				\overrel[>]{\eqref{eq:jscheme}}
					\tfrac{1}{\sqrt{\rp}}
					\tfrac{1}{\L_{\Omega}}
					\Bigl(
						\tfrac{P_j}{\tilde C_{\q}^\rp(\nicefrac{1}{\L_\Omega})}
					\Bigr)^{\frac{1-\q}{\q}}.
				\end{align*}
				Since \(\lam_i \leq \tfrac{1}{\L_{\Omega}}\) holds for all \(i=j+1,\dots,k\), it follows from \cref{thm:sublinear:Pk} that \(P_k\leq P_j\), and thus
				\begin{equation}\label{eq:Pk:bound:1a}
					\gamk*
				\geq
					\tfrac{1}{\sqrt{\rp}}
					\tfrac{1}{\L_{\Omega}}
					\Bigl(
						\tfrac{P_k}{\tilde C_{\q}^\rp(\nicefrac{1}{\L_\Omega})}
					\Bigr)^{\frac{1-\q}{\q}}.
				\end{equation}

			\item
				Alternatively, it holds that \(\lam_{j} \leq \tfrac{1}{\L_{\Omega}}\) for all \(j=n_{2,k},\dots,k\), and in particular by virtue of \cref{thm:sublinear:Pk<Dk} we have that \(P_k\leq P_{n_{2,k}}\).
				Arguing as before,
				\begin{equation}\label{eq:Pk:bound:1b}
					\gamk*
				=
					\gam_{n_{2,k}+1}
					\prod*_{i=n_{2,k}+1}^{k+1}\rho_i
				\overrel[\geq]{\eqref{eq:prodK1}}
					\tfrac{1}{\sqrt{\rp}}
					\gam_{n_{2,k}+1}
				\overrel[\geq]{\cref{claim:Pk:bound:2}}[3pt]
					\tfrac{1}{\sqrt{2\rp}\L_\Omega}
					\left(
						\tfrac{P_{n_{2,k}}}{\tilde C_{\q}^\rp(\lam_{\rm min})}
					\right)^{\frac{1-\q}{\q}}
				\geq
					\tfrac{1}{\sqrt{2\rp}\L_\Omega}
					\left(
						\tfrac{P_k}{\tilde C_{\q}^\rp(\lam_{\rm min})}
					\right)^{\frac{1-\q}{\q}}.
				\end{equation}
			\end{itemize}
			Combining \eqref{eq:Pk:bound:1a} and \eqref{eq:Pk:bound:1b}
			\[
				\gamk*
			\geq
				\min\set{
					\frac{1}{\tilde C_{\q}^\rp(\nicefrac{1}{\L_\Omega})^{\frac{1-\q}{\q}}}
				,~
					\frac{1}{\sqrt{2}\tilde C_{\q}^\rp(\lam_{\rm min})^{\frac{1-\q}{\q}}}
				}
				\frac{P_k^{\frac{1-\q}{\q}}}{\sqrt{\rp}\L_{\Omega}}
			=
				\frac{1}{\sqrt{2\rp}\L_{\Omega}\tilde C_{\q}^\rp(\lam_{\rm min})^{\frac{1-\q}{\q}}}
				P_k^{\frac{1-\q}{\q}},
			\]
			where the identity uses the fact that the minimum is attained at the first element, having \(\tilde C_{\q}^\rp(\nu)\) decreasing in \(\nu>0\) and \(\frac{1}{\L_\Omega}\geq\lam_{\rm min}=\frac{1}{\sqrt{2}\rho_{\rm max}\L_\Omega}\) (since \(\rho_{\rm max}\geq1\)).
		\end{claims}

		Finally, combining \cref{claim:Pk:bound:1,claim:Pk:bound:2} and noting that
		\[
		\frac{1}{\sqrt{2\rp}\L_\Omega}
		\left(
			\frac{P_k}{\tilde C_{\q}^\rp(\lam_{\rm min})}
			\right)^{\frac{1-\q}{\q}}
		=
			\frac{1}{\sqrt{2\rp}\L_\Omega^{\nicefrac{1}{\q}}}
			\left(
				\frac{1}{
					\sqrt{2\U_1(x^\star)}
					\bigl(\sqrt2 \rho_{\rm max}+1\bigr)
				}
			\right)^{\frac{1-\q}{\q}}
			P_k^{\frac{1-\q}{\q}},
		\]
		we conclude that
		\[
			\gamk*
		\geq
			\begin{ifcases}
				\displaystyle
				\frac{1}{
					\U_1(x^\star)^{\frac{1-\q}{2\q}}
					C(\rp,\q)^{\frac{1}{\q}}
				}
				P_k^{\frac{1-\q}{\q}}  & \text{(\(K_2\neq\emptyset\) and) } k \geq\min K_2,
			\\[15pt]
				\bigl(1 + \tfrac{1}{\rp}\bigr)^{\frac{k}{2}} \gam_0  \otherwise
			\end{ifcases}
		\]
		holds for any \(k\in\N\),
		where
		\[
			C(\rp,\q)
		=
			\sqrt{2}\L_\Omega
			\sqrt{\rp}^{\q}
			\bigl(
				1+\sqrt{2}\rho_{\rm max}
			\bigr)^{1-\q}
		\]
		is as in the statement.
		Denoting \(k_0=\min K_2-1\) if \(K_2\neq\emptyset\) and \(0\) otherwise, the sum of stepsizes can be lower bounded by
		\begin{align*}
		\textstyle
			\sum_{k=1}^{K+1} \gamk
		=
			\sum_{k=1}^{k_0} \gamk  + \sum_{k=k_0}^{K+1} \gamk
		\geq{} &
			\gam_0\sum_{k=1}^{k_0} (1 + \tfrac{1}{\rp})^{\frac{k}{2}}
			+
			\frac{1}{
				\U_1(x^\star)^{\frac{1-\q}{2\q}}
				C(\rp,\q)^{\frac{1}{\q}}
			}
			\sum_{k=k_0}^{K+1} P_{k-1}^{\frac{1-\q}{\q}}
		\\
		\geq{} &
			\gam_0\sum_{k=1}^{k_0} (1 + \tfrac{1}{\rp})^{\frac{k}{2}}
			+
			\frac{K+1-k_0}{
				\U_1(x^\star)^{\frac{1-\q}{2\q}}
				C(\rp,\q)^{\frac{1}{\q}}
			}
			\bigl(\min_{k \leq K}P_k\bigr)^{\frac{1-\q}{\q}}
		\\
		\geq{} &
			\gam_0 k_0
			+
			\frac{K+1-k_0}{
				\U_1(x^\star)^{\frac{1-\q}{2\q}}
				C(\rp,\q)^{\frac{1}{\q}}
			}
			\bigl(\min_{k \leq K}P_k\bigr)^{\frac{1-\q}{\q}}
		\\
		\geq{} &
			\min\set{
				\gam_0,
			~
				\frac{1}{
					\U_1(x^\star)^{\frac{1-\q}{2\q}}
					C(\rp,\q)^{\frac{1}{\q}}
				}
				\bigl(\min_{k \leq K}P_k\bigr)^{\frac{1-\q}{\q}}
			}
			(K+1).
		\end{align*}
		Therefore, in light of \cref{thm:sumgamk}, for every \(K\geq1\) we have
		\[
			\U_1(x^\star)
		\geq
			\min_{k\leq K}P_k
			\,
			\sum_{k=1}^{K+1}\gamk
		\geq
			\min\set{
				\gam_0\min_{k \leq K}P_k,
				\frac{1}{
					\U_1(x^\star)^{\frac{1-\q}{2\q}}
					C(\rp,\q)^{\frac{1}{\q}}
				}
				\bigl(\min_{k \leq K}P_k\bigr)^{\frac{1}{\q}}
			}
			(K+1).
		\]
		Equivalently, for every \(K\geq1\) it holds that
		\[
			\text{either}\quad
			\min_{k \leq K}P_k
		\leq
			\frac{
				\U_1(x^\star)
			}{
				\gam_0(K+1)
			}
		\quad\text{or}\quad
			\min_{k \leq K}P_k
		\leq
			\frac{
				\U_1(x^\star)^{\frac{1+\q}{2}}
				C(\rp,\q)
			}{
				(K+1)^{\q}
			}.
		\]
		Further using the fact that \(\U_1(x^\star)\leq\U_0(x^\star)\) by \cref{lemma:main:inequality} results in the claimed bound.
	\end{appendixproof}
\end{WhereToPutThis}

	\section{\protect Implementation details of AC-FGM}\label{sec:algs}
		In this section we describe the specific implementation of%
		\noindent{\  the auto-conditioned fast gradient method (AC-FGM) \cite{li2023simple}},{}  which in our notation{}  reads{}  		\begin{align*}
			z^{k+1}
		={}&
			\prox_{\gamk*g}(y^k-\gamk*\nabla f(x^k))
		\\
			y^{k+1}
		={}&
			(1-\beta_{k+1})y^k + \beta_{k+1} z^{k+1}
		\\
			x^{k+1}
		={}&
			(z^{k+1} + \tau_{k+1}x^k) / (1 + \tau_{k+1}).{}  		\end{align*}
		Regarding the positive sequences \((\gamma_k)_{k \in \N}, (\beta_k)_{k \in \N}, (\tau_k)_{k \in \N}\), we use the update rule described in \cite[Cor. 3]{li2023simple} and as such in Corollary 2 of the same paper.
		We choose \(\beta_1 = 0\) and \(\beta_k = \beta = \tfrac{1-\sqrt{3}}{2}\) for all \(k \geq 2\),
		ensure that \(\gamma_1 \in [\tfrac{\beta}{4(1-\beta)c_1}, \tfrac{1}{3c_1}]\) and set \(\gamma_2 = \tfrac{\beta}{2c_1}\) and
		\(\gamk* = \min\{\tfrac{\tau_{k-1}+1}{\tau_{k}}\gamk, \tfrac{\beta \tau_{k}}{4c_k}\}\) for all \(k \geq 2\).
		Finally, \(\tau_1 = 0\), \(\tau_2 = 2\) and \(\tau_{k+1}= \tau_k + \tfrac{\alpha}{2} + \tfrac{2(1-\alpha)\gamk c_k}{\beta \tau_k}\), for \(k \geq 2\) and some \(\alpha \in [0, 1]\).
		We chose \(\alpha = 0\), since this configuration consistently outperfomed the others in our experiments.
		The sequence \((c_k)_{k \in \N}\) is the so-called local Lipschitz estimate and is defined as in \cite[Eq. (3.9)]{li2023simple}:
		\[
					c_k
				=
					\begin{ifcases}
						\tfrac{\sqrt{\|x^1-x^0\|^2\|\nabla f(x^1)-\nabla f(x^0)\|^2 + (\epsilon/4)^2}-\epsilon/4}
						{\|x^1-x^0\|^2}
						& k = 1,
					\\[3pt]
					\tfrac{\|\nabla f(x^k)-\nabla f(x^{k-1})\|^2}
					{2[f(x^{k-1})-f(x^k) - \innprod{\nabla f(x^k)}{x^{k-1}-x^k}] + \epsilon / \tau_{k}}
					\otherwise.
					\end{ifcases}
				\]
		where \(\epsilon\) is the predefined desired accuracy of the algorithm.

\ifarxiv
	\clearpage
	\phantomsection
	\addcontentsline{toc}{section}{References}
	\bibliographystyle{plain}
	\bibliography{Bibliography.bib}
\fi

\end{document}